\documentclass[12pt, reqno]{amsart}

\usepackage{fullpage}
\usepackage{amsmath}
\usepackage{amssymb}
\usepackage{dsfont}
\usepackage{tikz}
\usepackage{tikz-3dplot}
\usepackage[english]{babel}
\usetikzlibrary{calc}
\usepackage{pgfplots}
\pgfplotsset{compat=1.18}
\usepackage{microtype}
\usepackage{mathtools}
\usepackage{extarrows}
\usepackage{booktabs}
\usepackage{hyperref}
\usepackage{cleveref}
\usepackage{subcaption}
\usetikzlibrary{backgrounds,datavisualization.formats.functions}

\newcommand{\KK}{\mathbb{K}}
\newcommand{\NN}{\mathbb{N}}

\newcommand{\RR}{\mathbb{R}}

\newcommand{\ZZ}{\mathbb{Z}}

\newcommand{\sphere}{\mathbb{S}}
\newcommand{\mc}{\mathcal}

\newcommand{\bs}{\boldsymbol}

\newcommand{\conv}{\operatorname{conv}}
\newcommand{\cone}{\text{cone}}
\newcommand{\Ehr}{\text{Ehr}}
\newcommand{\Hilb}{\text{Hilb}}
\newcommand{\initial}{\text{in}}
\newcommand{\link}{\text{lk}}

\newcommand{\Skel}{\text{Skel}}

\newtheorem{lemma}{Lemma}[section]
\newtheorem{theorem}[lemma]{Theorem}
\newtheorem{corollary}[lemma]{Corollary}
\newtheorem{proposition}[lemma]{Proposition}

\theoremstyle{definition}

\newtheorem{example}[lemma]{Example}

\newtheorem{remark}[lemma]{Remark}

\definecolor{brickred}{rgb}{0.8, 0.25, 0.33}

\numberwithin{equation}{section}

\title{Boundary $h^\ast$-vectors and unimodular triangulations}
\author[M.~Juhnke]{Martina Juhnke}
\address{Universit\"at Osnabr\"uck, Fakult\"at f\"ur Mathematik,
  Albrechtstraße 28a, 49076 Osnabr\"uck, Germany}
\email{martina.juhnke@uni-osnabrueck.de}

\author[S.~Schlie]{Steffen Schlie}
\address{Universit\"at Osnabr\"uck, Fakult\"at f\"ur Mathematik,
  Albrechtstraße 28a, 49076 Osnabr\"uck, Germany}
\email{steffen.schlie@uni-osnabrueck.de}

\begin{document}
\maketitle
\begin{abstract}
We study the Ehrhart $h^\ast$-polynomial of (the boundary of) a lattice polytope via regular unimodular triangulations and Gr\"obner degenerations of toric ideals. Our main result is a boundary analogue of the well-known Sturmfels correspondence. This allows us to connect the boundary $h^\ast$-polynomial to the $h$-polynomial of any regular unimodular triangulation, in analogy to the classical  Betke-McMullen Theorem. 

Providing a direct link between Ehrhart theory and the face enumeration of simplicial complexes, we then transfer structural results from the theory of simplicial polytopes to the setting of lattice polytopes. In particular, we derive general Dehn-Sommerville-type relations between $h^\ast(P)$ and $h^\ast(\partial P)$. Under the additional assumption of $\partial P$ admitting a regular unimodular triangulation, we recover old and prove new characterization results concerning symmetry or unimodality, as well as upper and lower bounds for coefficient-wise differences within $h^\ast(P)$.
\end{abstract}

%\tableofcontents

%%%--------------------------------------------------------------------------------%%%
%%%                           S E C T I O N     1                                  %%%
%%%--------------------------------------------------------------------------------%%%
\section{Introduction}

The Ehrhart polynomial $L_P(m)$ of a $d$-dimensional lattice polytope $P$ encodes the number of lattice points in its $m$\textsuperscript{th} integer dilation. It was shown by Ehrhart \cite{Ehrhart1962} that its generating series, called Ehrhart series, can be written as a rational function:
\[
\Ehr_P(z) = 1+\sum_{m\geq 1} L_P(m)z^m = \frac{h_P^\ast(z)}{(1-z)^{d+1}} = \frac{h_0^\ast + h_1^\ast z + \cdots + h_d^\ast z^d}{(1-z)^{d+1}}.
\]
A central question in Ehrhart theory asks for general combinatorial interpretations of the coefficient vector $h^\ast(P)= (h^\ast_0,h_1^\ast,\dots,h_d^\ast)$. As of today, there are only partial answers to this question. For instance, if $P$ is a lattice $d$-simplex, then $h_k^\ast$ indeed counts the number of lattice points with last coordinate equal to $k$ lying in a certain half-open parallelepiped \cite[Corollary 3.11]{Beck2010}. Another prime example, formulated by Betke and McMullen \cite{BetkeMcMullen1985}, holds for any lattice polytope admitting a unimodular triangulation, stating that $h^\ast_P(z)$ is in fact equal to the $h$-polynomial of any unimodular triangulation (see \Cref{th:betke-mcmullen} below).

\begin{example}[\textit{Lattice polygons}]\label{exa:latticePolygons}
    Let $P = \text{conv}\{(0,0),(-1,1),(1,3),(4,2),(5,1),(2,0)\}$. An immediate consequence from Pick's famous theorem \cite{Pick1899} is that every lattice polygon has a unimodular triangulation (see, for example, \cite[Proposition 1.1]{HaasePaffenholzPiechnikSantos2021}). One choice of such a triangulation $\mc{T}$ is depicted in Figure \ref{fig:Examples-Polygon-CrossPolytope}. Using the relation (\ref{eq:relation-h-f}) below, one computes 
    \[
    h_\mc{T}(z) = (1-z)^3 + 16z(1-z)^2 + 36z^2(1-z) + 22z^3 = 1+13z+8z^2.
    \]
    Once again, observing that $P$ has an area of 11, contains eight interior and eight boundary lattice points, Pick's theorem helps computing $h^\ast_P(z) = 1+13z+8z^2$. 
\end{example}

\begin{example}[\textit{Cross-polytopes}]\label{exa:Crosspolytope}
    Let $P = \{x \in \RR^3:|x_1|+|x_2|+|x_3| \leq 1\}$ be the octahedron in $\RR^3$. Its $h^\ast$-polynomial is given by $h^\ast_P(z) = (1+z)^3$ (see \cite[Theorem 2.7]{Beck2010}). One can triangulate $P$ unimodularly by coning over all facets of $P$ with the origin as apex (see Figure \ref{fig:Examples-Polygon-CrossPolytope}). The $h$-polynomial of this triangulation is $h_\mc{T}(z) = 1+3z+3z^2+z^3$.   
\end{example}

\begin{figure}[t]
    \begin{minipage}{0.45\textwidth}
        \centering
        \includegraphics[width=0.95\textwidth]{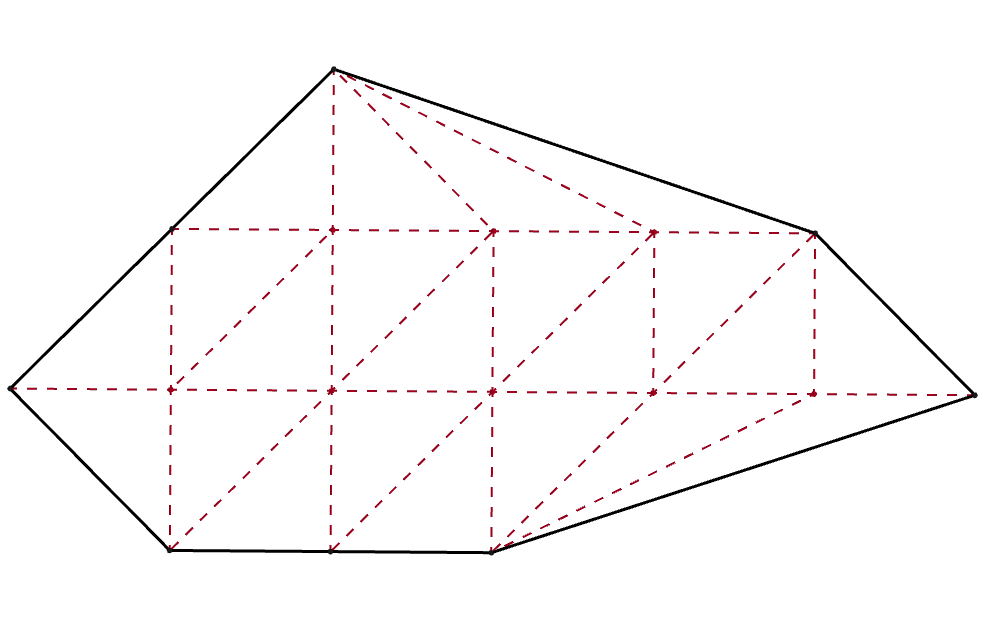}
    \end{minipage}
    \hspace{1em}
    \begin{minipage}{0.45\textwidth}
        \centering
        \includegraphics[width=1.05\textwidth]{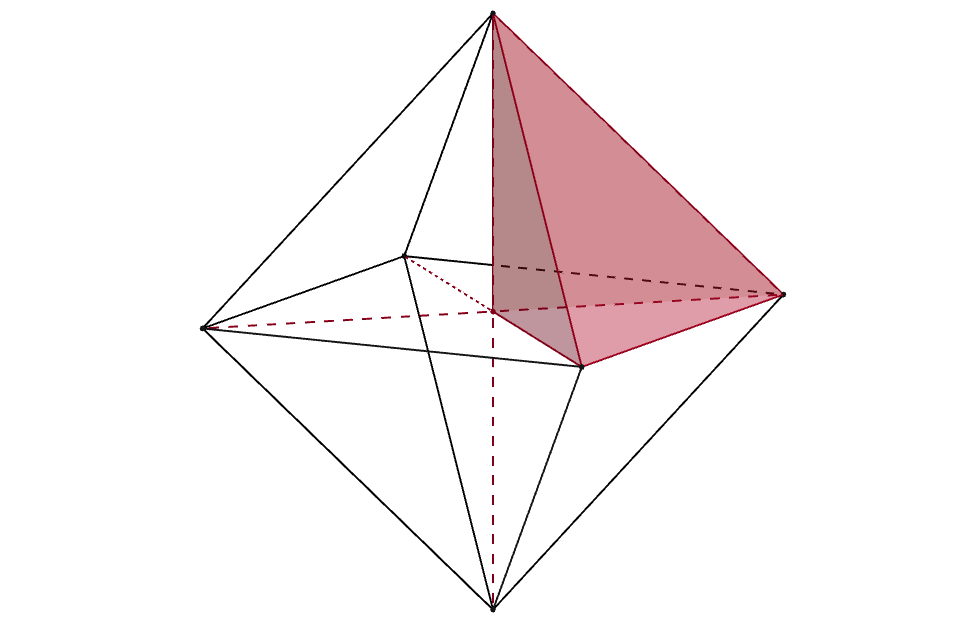}
    \end{minipage}
    \caption{Unimodular triangulations of a lattice polygon (\Cref{exa:latticePolygons}) and the $3$-dimensional octahedron (\Cref{exa:Crosspolytope}).}
    \label{fig:Examples-Polygon-CrossPolytope}
\end{figure}
\vspace{1em}
In recent years, considerable attention has been devoted to the question whether the Ehrhart $h^\ast$-polynomial has any kind of decomposition property. It was proven by Betke-McMullen \cite{BetkeMcMullen1985} as well as Stapledon \cite{Stapledon2009} for lattice polytopes (and by Beck-Braun-Vindas-Meléndez \cite{BeckBraunVindasMelendez2022} for rational polytopes) that $h_P^\ast(z)$ can be decomposed using palindromic polynomials with nonnegative coefficients. In terms of coefficients, this generalizes Stanley's famous nonnegativity theorem \cite{Stanley1980B}. Moreover, recent results by Bajo and Beck \cite{BajoBeck2023} showed that for lattice polytopes containing an interior lattice point, the $h^\ast$-polynomial of the boundary of $P$ naturally appears as a constituent part in these decompositions. In particular, $h^\ast_P(z)$ and $h^\ast_{\partial P}(z)$ coincide when $P$ is reflexive (see \cite[Corollary 5.4]{BajoBeck2023}). 

This raises the following question: What information about $h^\ast_P(z)$ is already determined by $h^\ast_{\partial P}(z)$ and which properties of $\partial P$ impose substantial restrictions on the coefficients $h_0^\ast,h_1^\ast,\dots,h_d^\ast$? Although not explicitly stated, Stapledon’s work \cite{Stapledon2009} yields, almost as a byproduct, a boundary analogue of the Betke–McMullen theorem \cite{BetkeMcMullen1985}. If $\partial P$ admits a unimodular triangulation, then $h^\ast_{\partial P}(z)$ is exactly the $h$-polynomial of this triangulation (see \Cref{th:stapledon} below). 

\begin{example}[\textit{Boundary of lattice polygons} \cite{BajoBeck2023}]
    Let $P = \text{conv}\{(0,0),(0,2),(2,0),(3,3)\}$ and denote by $\partial \mc{T}$ a unimodular triangulation of the boundary $\partial P$. Then 
    \[
    h^\ast_{\partial P}(z) = 1+4z+z^2 = (1-z)^2 + 6z(1-z)+6z^2 = h_{\partial \mc{T}}(z).
    \]
\end{example}

Both the proofs of Betke-McMullen's general theorem and Stapledon's boundary version are combinatorial in nature. They use local $h$-vectors of a triangulation, their nonnegativity and the Dehn-Sommerville relations. However, it has been known for quite some time that their results admit a natural interpretation in terms of Gr\"obner degenerations, semigroup algebras and Stanley-Reisner theory. This was mainly initiated by a prime correspondence going back to Sturmfels \cite{Sturmfels1996}. It is worth mentioning that this interpretation only works for triangulations that are regular (see \Cref{sect:EhrhartComAlg}). Our first contribution is threefold: Starting off, we prove a special case of Betke-McMullen's result through the aforementioned lens of commutative algebra. Though it might be well-known to experts, we have not found such a proof in the literature. We  are then able to prove an analogue of Sturmfels' correspondence between regular (unimodular) triangulations of boundary complexes and (square-free) initial ideals of the boundary version of usual toric ideals. While this is the main theorem of the first part of our work, it immediately implies Stapledon's theorem for regular unimodular triangulations.

In the second part of this article we aim at contributions to another unifying objective in Ehrhart theory~--~characterizing the set of vectors in $\RR^{d+1}$ that appear as coefficient vectors of $h^\ast_P(z)$ for lattice $d$-polytopes. This is a very difficult problem not yet fully understood, even for low dimensions (see \cite[Problem 10.21]{Beck2010} for $d=3$). Instead of considering the entire set of polytopes, a common approach is to study smaller classes that share certain properties. We will restrict ourselves (with \Cref{th:ds-lattice-polytopes} being the only exception) to lattice polytopes (or their boundary complexes) that have a regular unimodular triangulation. Then, one can investigate whether the $h^\ast$-vectors of these polytopes meet certain properties like (shifted) symmetry, unimodality, log-concavity or real-rootedness, among others. For further information on examples and counterexamples regarding these properties we refer to the latest work of Ferroni and Higashitani \cite{FerroniHigashitani2024}.

One prominent example that serves as a motivation for our work are the Dehn-Sommerville relations. Formulated by Klee \cite{Klee1964} for homology manifolds without boundary and later extended to homology manifolds with boundary by Gr\"abe \cite{Graebe1987}, they govern a deep structural layer in the theory of $h$-vectors of simplicial polytopes.% boundary complexes of lattice polytopes. 
%Following a proof in \cite{NovikSwartz09}, 
We are able to derive general Dehn-Sommerville relations of differences between coefficients of $h^\ast_P(z)$ and $h^\ast_{\partial P}(z)$. It should not come as a total surprise that these relations hold for any lattice $d$-polytope. However, the existence of a regular (unimodular) triangulation $\mathcal{T}$ of $\partial P$ guarantees that $\mathcal{T}$ is polytopal, i.e.\ $\partial \mathcal{T} \cong \partial Q$ (as simplicial complexes) for a simplicial polytope $Q$ (see \cite{DeLoeraRambauSantos2010}). This bridge allows us to translate results for $h$-vectors of simplicial polytopes to regular boundary triangulations of lattice polytopes. The most prominent examples are the  $g$-theorem (see \cite{BilleraLee1981,Stanley1980}) and the Generalized Lower Bound Theorem (see \cite{McMullenWalkup1971,MuraiNevo2013,Stanley1980}) for simplicial polytopes. The former imposes immediate symmetry and unimodality conditions as well as $M$-sequence properties on $h^\ast(\partial P)$, recovering some results by Stapledon \cite[Theorem\ 2.20]{Stapledon2009}. The latter allows us to characterize equality of consecutive coefficients of $h^\ast(\partial P)$ via stackedness of a regular unimodular triangulation $\mathcal{T}$ of $\partial P$, leading to symmetry of $h^\ast(P)$ between $h_r^\ast(P)$ and $h^\ast_{d-r+1}(P)$ for some $r \leq \lfloor d/2 \rfloor$. Our last contribution deals with upper and lower bounds on the differences $h^\ast_j(P) - h^\ast_{d-j+1}(P)$ by moving to the boundary of $P$. 

The structure of the paper is as follows: \Cref{sect:Prelim} contains all necessary notation and concepts from Ehrhart theory, commutative algebra and homology theory that are used throughout. \Cref{sect:EhrhartComAlg} provides the proofs of (the special cases of) Betke-McMullen's theorem as well as the boundary case going back to Stapledon. In \Cref{sect:Applic}, we discuss implications of these results for the $h^\ast$-vector of $P$. While most results require $\partial P$ or $P$ to admit a regular unimodular triangulation, some hold in broader generality for all lattice polytopes.

%%%--------------------------------------------------------------------------------%%%
%%%                           S E C T I O N     2                                  %%%
%%%--------------------------------------------------------------------------------%%%
\section{Preliminaries}\label{sect:Prelim}

\textit{Simplicial complexes and Stanley-Reisner ideals.} An (abstract) \emph{simplicial complex} $\mc{K}$ on a (finite)  set $V$ is a collection of subsets (also called \emph{faces}) of $V$ that is closed under inclusion. The \emph{dimension} of  a face $F$ is defined as $\dim F\coloneqq \#F-1$, and $\dim \mc{K}\coloneqq \max(\dim F~:~F\in \mc{K})$ is called the \emph{dimension} of $\mathcal{K}$. The \emph{link} of a face $F \in \mc{K}$ is defined as $\link_\mc{K}(F) \coloneqq \{G \in \mc{K}~:~ F \cup G = \mc{K} \text{ and } F \cap G = \emptyset\}$. 
A natural invariant of a $d$-dimensional simplicial complex $\mc{K}$ is its so-called \emph{$h$-polynomial} $h_{\mathcal{K}}(z)$ defined as
\begin{align}\label{eq:relation-h-f}
    h_\mc{K}(z) \coloneqq\sum_{k=0}^{d+1}h_k(\mc{K})z^k\coloneqq  \sum_{k=-1}^d f_k(\mc{K})z^{k+1}(1-z)^{d-k},
\end{align}
where $f_{-1}(\mc{K}) = 1$ and $f_k(\mc{K})$ is the number of $k$-dimensional faces of $\mathcal{K}$. The vector $h(\mc{K})=(h_0(\mc{K}),h_1(\mc{K}),\ldots,h_d({\mc{K}}))$ is called the \emph{$h$-vector} of $\mc{K}$.

%As already pointed out, a triangulation of a lattice polytope  $P\subseteq \mathbb{Z}^d$ or its boundary can be seen as the geometric realization of a simplicial complex on the set $P\cap \mathbb{Z}^d$. 
Let $\KK$ be a field. Given a simplicial complex  $\mc{K}$ on $V$, it is natural to consider its so-called \emph{Stanley-Reisner ideal} $I_{\mc{K}}$, which is the squarefree monomial ideal in $\KK[x_v~:~v\in V]$ whose generators correspond to non-faces of $\mc{K}$:
$$I_\mathcal{K} = \langle \bs{x}^\sigma=\prod_{v\in \sigma}x_v \;:\;\sigma \notin \mc{K}\rangle\subseteq \KK[x_v~:~v\in V].$$
%Each subset $\sigma \subseteq [n]$ is identified with its squarefree vector in $\{0,1\}^n$, which has entry $1$ in the $i$-th entry if $i \in \sigma$ and $0$ otherwise. This convention allows us to write $\bs{x}^\sigma = \prod_{i=1}^n x_i$ for a monomial in the polynomial ring $R$. We define the Stanley-Reisner ideal of the simplicial complex $\mathcal{T}$ as the squarefree monomial ideal
%$$I_\mathcal{T} = \langle \bs{x}^\sigma \;|\;\sigma \notin \mathcal{T}\rangle$$
%generated by monomials corresponding to non-faces of $\mathcal{T}$. 
The \emph{Stanley-Reisner ring} of a $d$-dimensional simplicial complex  $\mc{K}$ is the quotient ring $\KK[\mc{K}]\coloneqq \KK[x_v~:~v\in V]/I_{\mc{K}}$. Its Hilbert series is given by
\begin{align}\label{eq:hilb-series-sr}
    \Hilb_{\KK[\mc{K}]}(z) = \frac{h_{\mc{K}}(z)}{(1-z)^{d+1}} = \frac{\sum_{i=0}^d h_i(\mc{K})z^i}{(1-z)^{d+1}}
\end{align}
(see e.g., \cite[Theorem 1.13, Corollary 1.15]{MillerSturmfels2005}).

\begin{example}
    Consider the simplicial complex $\mc{K}$ on $\{1,\ldots,6\}$ consisting of all subsets of the sets $\{1,2,3\},\{2,4\},\{3,4\},\{5,6\}$. A geometric realization is
    \begin{center}
        \begin{tikzpicture}[scale=0.9, vertex/.style={circle,fill=black,inner sep=1.8pt},
                            edge/.style={line width=1pt},face/.style={fill=gray!50}]
        % Left component vertices
        \coordinate (v1) at (0,0);
        \coordinate (v2) at (2,0);
        \coordinate (v3) at (1,1.8);
        \coordinate (v4) at (3.2,1.8); 
        % Right component vertices
        \coordinate (v5) at (5.8,0);
        \coordinate (v6) at (5.8,1.8);
        % Filled 2-simplex {1,2,3}
        \fill[face] (v1) -- (v2) -- (v3) -- cycle;
        % Left component edges
        \draw[edge] (v1) -- (v2);
        \draw[edge] (v1) -- (v3);
        \draw[edge] (v2) -- (v3);
        \draw[edge] (v3) -- (v4);
        \draw[edge] (v2) -- (v4);
        % Right component edge
        \draw[edge] (v5) -- (v6);
        % Vertices
        \node[vertex,label=below left:$1$] at (v1) {};
        \node[vertex,label=below:$2$] at (v2) {};
        \node[vertex,label=above:$3$] at (v3) {};
        \node[vertex,label=above:$4$] at (v4) {};
        \node[vertex,label=below:$5$] at (v5) {};
        \node[vertex,label=above:$6$] at (v6) {};
        \end{tikzpicture}
    \end{center}
  and its Stanley-Reisner ideal is equal to 
  \[
    I_\mc{K} = \langle x_1 x_4, x_1 x_5, x_1 x_6, x_2 x_5, x_2 x_6, x_3 x_5, x_3 x_6, x_4 x_5, x_4 x_6, x_2 x_3 x_4 \rangle.
  \]
\end{example}

\textit{Homology and simplicial manifolds.} A closed manifold $M$ is called \emph{triangulable} if there exists a simplicial complex $\mc{K}$, whose geometric realization is homeomorphic to $M$ (see Figure \ref{fig:Triangulable-Manifold}). We call such $\mc{K}$ a \emph{simplicial manifold}. Readers should note that in \Cref{subsect:differences}, \Cref{claim:santos} will  allow us to focus on a specific class of simplicial manifolds: simplicial polytopes. Their boundary complexes naturally give rise to a simplicial complex and, abusing notation, we will refer to their $h$-vector as the $h$-vector of the polytope. Luckily, the next two important results, which will be central in \Cref{sect:Applic}, hold for this class of polytopes. For a broader treatment of face enumeration of simplicial manifolds we refer the interested reader to the insightful survey by Klee and Novik \cite{KleeNovik16}.

The celebrated $g$-theorem of Stanley \cite{Stanley1980} (necessity) and Billera-Lee \cite{BilleraLee1981} (sufficiency) gives a complete characterization of $h$-vectors of simplicial polytopes. 

\vspace{1em}
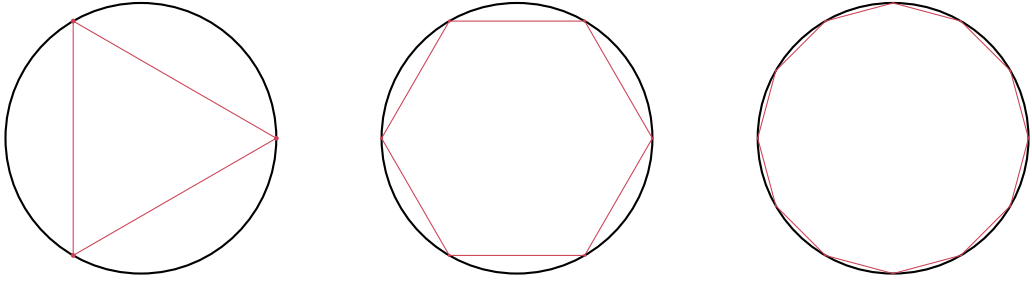
\begin{figure}[h]
    \centering
    \begin{tikzpicture}[scale=1]
        \def\r{1.8}
        % ---------------- Left: 3 points ----------------
        \begin{scope}[xshift=-5cm]
            \draw[thick] (0,0) circle (\r);
            \foreach \i in {0,...,2} {
                \coordinate (A\i) at ({\r*cos(120*\i)},{\r*sin(120*\i)});
            }
            \draw[brickred, thin] (A0)
                \foreach \i in {1,...,2} { -- (A\i) } -- cycle;
            \foreach \i in {0,...,2} {
                \fill[brickred] (A\i) circle (0.9pt);
            }
        \end{scope}
        % ---------------- Middle: 6 points ----------------
        \begin{scope}
            \draw[thick] (0,0) circle (\r);
            \foreach \i in {0,...,5} {
                \coordinate (B\i) at ({\r*cos(60*\i)},{\r*sin(60*\i)});
            }
            \draw[brickred, thin] (B0)
                \foreach \i in {1,...,5} { -- (B\i) } -- cycle;
            \foreach \i in {0,...,5} {
                \fill[brickred] (B\i) circle (0.6pt);
            }
        \end{scope}
        % ---------------- Right: 12 points ----------------
        \begin{scope}[xshift=5cm]
            \draw[thick] (0,0) circle (\r);
            \foreach \i in {0,...,11} {
                \coordinate (C\i) at ({\r*cos(30*\i)},{\r*sin(30*\i)});
            }
            \draw[brickred, thin] (C0)
                \foreach \i in {1,...,11} { -- (C\i) } -- cycle;
            \foreach \i in {0,...,11} {
                \fill[brickred] (C\i) circle (0.4pt);
            }
        \end{scope}
    \end{tikzpicture}
    \caption{Triangulations of the closed manifold $\sphere^1$ through polygons with $n=3,6 \text{ and } 12$ vertices.}
    \label{fig:Triangulable-Manifold}
\end{figure}

\begin{theorem}[The $g$-theorem \cite{BilleraLee1981,Stanley1980}]\label{th:g-theorem}
    A vector $\bs{h} = (h_0,h_1,\dots,h_d) \in \ZZ^{d+1}$ is the $h$-vector of a simplicial $d$-polytope $\mc{K}$ if and only if
    \begin{enumerate}
        \item[(i)] $h_j = h_{d-j}$ for all $j$,
        \item[(ii)] $1 = h_0 \leq h_1 \leq \cdots \leq h_{\lfloor d/2 \rfloor}$, and
        \item [(iii)] the numbers $\bs{g} \coloneqq (g_0,g_1,\dots,g_{\lfloor d/2\rfloor})$, where $g_j \coloneqq h_j - h_{j-1}$, form an $M$-sequence.
    \end{enumerate}
\end{theorem}
We will provide the definition of an $M$-sequence in \Cref{subsect:differences}. 
\noindent The Generalized Lower Bound Theorem characterizes the equality cases  in (ii) of \Cref{th:g-theorem} as follows:
\begin{theorem}[GLBT \cite{McMullenWalkup1971,MuraiNevo2013,Stanley1980}]\label{th:glbt}
    Let $\mc{K}$ be a simplicial $d$-polytope. Then
    \begin{enumerate}
        \item[(i)] $h_0(\partial\mc{K}) \leq h_1(\partial\mc{K}) \leq \cdots \leq h_{\lfloor d/2 \rfloor}(\partial\mc{K})$, and 
        \item[(ii)] $h_r(\partial\mc{K}) = h_{r-1}(\partial\mc{K})$ for some $r \leq \lfloor d/2 \rfloor$ if and only if $\mc{K}$ is $(r-1)$-stacked. 
    \end{enumerate}
\end{theorem}
We will provide the definition of \emph{stackedness} in \Cref{subsect:differences} as well. \\

\textit{Lattice polytopes, triangulations and Ehrhart theory.} Let $P \subset \RR^d$ be a lattice polytope, i.e., the convex hull of finitely many points in $\ZZ^d$. The \emph{dimension} of $P$ is defined as the dimension of its affine hull and we assume that $P$ is full-dimensional. A $d$-dimensional lattice polytope will be referred to as a \emph{$d$-polytope}. The union of all facets of $P$ forms the \emph{boundary} $\partial P$ and $P^\circ \coloneqq P\setminus \partial P$ is the \emph{interior} of $P$. We write 
$$\mathcal{A} = \{\bs{\alpha}_1,\dots,\bs{\alpha}_n\} = P \cap \ZZ^d$$ 
for the set of lattice points contained in $P$ and $V \subseteq \mathcal{A}$ for the subset of vertices.
Moreover, we denote by $\cone(P)$ the  \emph{homogenization} of $P$ which is obtained by embedding $P$ in $\RR^{d+1}$ at height $x_{d+1}=1$ and taking the cone over $P\times \{1\}$ with apex the origin, i.e.,
\[
\cone(P)\coloneqq \left\{\sum_{\bs{v}\in V}\lambda_{\bs{v}} (\bs{v},1)~:~\lambda_{\bs{v}}\geq 0 \text{ for }\bs{v}\in V \right\}.
\]
Frequently, we will also refer to $\cone(P)$ simply as the cone over $P$.
A \emph{subdivision} of $P$ is a finite collection $\mathcal{S}$ of polytopes $P_1,\dots,P_m$ such that vertices of each $P_i$ are drawn from $V$, the union of these polytopes is equal to $P$ and $P_i \cap P_j$ is a common face of both $P_i$ and $P_j$ for $i\neq j$. If each $P_i$ is a simplex, i.e., its vertices are affinely independent and $\mathcal{S}$ is closed under taking faces, then $\mathcal{S}$ is called a \emph{triangulation}. We usually use  $\mathcal{T}$ to denote a triangulation. Note that the second condition in the definition of a triangulation $\mathcal{T}$ ensures that $\mathcal{T}$ can be interpreted as the geometric realization of the simplicial complex whose faces are the vertex sets of the simplices in $\mathcal{T}$. In order to emphasize which polytope or face $\mc{S}$ subdivides, we occasionally write $\mc{S}(P)$ and $\mc{S}(F)$, respectively. Subdivisions that can be obtained by lifting the vertices of $P$ to $\RR^{d+1}$ are called \emph{regular}. More precisely, the vertices $\bs{v}_1,\dots,\bs{v}_n \in V$ are lifted to $(\bs{v}_1,\omega(\bs{v}_1)),\dots,(\bs{v}_n,\omega(\bs{v}_n)) \in \RR^{d+1}$ by a weight function $\omega:V \to \RR$. The subdivision $\mathcal{S}$ is then induced by projecting the lower facets (i.e., the facets whose outer normal vector has a negative last coordinate) of $\text{conv}((\bs{v}_1,\omega(\bs{v}_1)),\dots,(\bs{v}_n,\omega(\bs{v}_n)))$ to $\RR^d \times \{0\}$. A weight function $\omega$ is called \emph{generic} if the induced regular subdivision is a triangulation. Note that different weight functions may induce the same subdivision, in fact, regular subdivisions are naturally parameterized by equivalence classes of weight functions. Since $\omega$ can be also viewed as a vector, we sometimes write $\omega \in \RR^{|V|}$ in short. Given a subdivision $\mathcal{S}$, we call a polytope $F\in \mathcal{S}$ a \emph{boundary face} if $F\subseteq G$ for a $(d-1)$-polytope $G\in \mathcal{S}$  that is contained in a unique $d$-polytope $H\in \mathcal{S}$. The restriction of $\mathcal{S}$ to $\partial P$ gives the \emph{boundary complex}
$$\partial \mathcal{S} = \{F\in \mathcal{S} ~ : ~F \text{ is a boundary face of } \mathcal{S}\}.$$
As such, it is natural to consider its so-called \emph{$h$-polynomial} $h_{\mathcal{T}}(z)$ defined as 
\begin{align}\label{eq:relation-h-f}
    h_\mathcal{T}(z) \coloneqq\sum_{k=0}^{d+1}h_k(\mathcal{T})z^k\coloneqq  \sum_{k=-1}^d f_k(\mathcal{T})z^{k+1}(1-z)^{d-k},
\end{align}
where $f_{-1}(\mathcal{T}) = 1$ and $f_k(\mathcal{T})$ is the number of $k$-dimensional faces of $\mathcal{T}$. A triangulation $\mathcal{T}$ is \emph{unimodular} if each $d$-simplex has normalized volume $1$. In this case, $P$ is known to have the \emph{integer decomposition property} (IDP, for short) which means that  for every $\bs{z} \in mP \cap \ZZ^d$ there exist  $\bs{y}_1,\dots,\bs{y}_m \in P \cap \ZZ^d$ such that $\bs{z} = \bs{y}_1+\cdots+\bs{y}_m$.

For a positive integer $m$, let $L_P(m)$ denote the number of integer points in $mP$. 
Ehrhart's theorem \cite{Ehrhart1962} says that the generating series of $L_P(m)$, the so-called \emph{Ehrhart series}, is given by the following rational function:
$$\Ehr_P(z) \coloneqq 1+ \sum_{m \geq 1} L_P(m)z^m = \frac{h^\ast_P(z)}{(1-z)^{d+1}}.$$
The numerator polynomial, which is the so-called \emph{$h^\ast$-polynomial} of $P$, has degree at most $d$ and nonnegative coefficients \cite{Stanley1980}. Similarly, we define the $h^\ast$-polynomials of $\partial P$ and $P^\circ$ via 
\begin{align*}
    \Ehr_{\partial P}(z) \coloneqq 1+ \sum_{m \geq 1} L_{\partial P}(m)z^m = \frac{h^\ast_{\partial P}(z)}{(1-z)^d}
\end{align*} 
and
\begin{align*}
    \Ehr_{P^\circ}(z) \coloneqq \sum_{m \geq 1} L_{P^\circ}(m)z^m = \frac{h^\ast_{P^\circ}(z)}{(1-z)^{d+1}},
\end{align*}
where $L_{\partial P}(m)=|\partial (m P)\cap \mathbb{Z}^d|$ and $L_{P^\circ}(m)=|(mP)^\circ\cap \mathbb{Z}^d|$. 
Note that $h^\ast_{\partial P}(z)$ is a polynomial of degree less than or equal to $d-1$. For more information on decompositions of $h^\ast_P(z)$ into palindromic polynomials with nonnegative coefficients and the interplay between $h^\ast_P(z)$, $h^\ast_{\partial P}(z)$ and $h^\ast_{P^\circ}(z)$ we refer to the works of Bajo-Beck \cite{BajoBeck2023} and Stapledon \cite{Stapledon2009}.

Because we will refer to the following result several times, we state it here.
\begin{theorem}[Betke-McMullen \cite{BetkeMcMullen1985}]\label{th:betke-mcmullen}
    Let $P$ be a lattice $d$-polytope that admits a unimodular triangulation $\mc{T}$. Then $h^\ast_P(z) = h_\mc{T}(z)$. 
\end{theorem}

Since the boundary case was implicitly contained but not explicitly stated in Stapledon’s work \cite{Stapledon2009}, we are doing so here for the first time. To distinguish a triangulation of $\partial P$ from a triangulation of $P$, we use $\Delta$ instead of $\mc{T}$.
\begin{theorem}\label{th:stapledon}
    Let $P$ be a lattice $d$-polytope whose boundary $\partial P$ admits a unimodular triangulation $\Delta$. Then $h^\ast_{\partial P}(z) = h_\Delta(z)$. 
\end{theorem}

\textit{Semigroups and semigroup algebras.} Let $\KK$ be a field and let $\mathcal{A}\subseteq \ZZ^d$ be a set of lattice points. We consider a polynomial ring $R\coloneqq \KK[\bs{x}]=\KK[x_{\bs{\alpha}}~:~\bs{\alpha}\in \mathcal{A}]$, where each variable corresponds to a lattice point in $\mathcal{A}$.  The \emph{affine semigroup} associated to $\mathcal{A}$ is the  finitely generated subsemigroup $S = \NN\{\mathcal{A}\} \subset \ZZ^d$ that contains $\mathbf{0}$. % Throughout this paper we will always assume that we have the same number of variables in $R$ as lattice points in $\mathcal{A}$. 
The \emph{semigroup ring} $\KK[S]$ is the $\KK$-algebra with $\KK$-basis $\{\bs{t}^{\bs{a}} = t_1^{a_1}\cdots t_d^{a_d}~:~ \bs{a} \in S\}$ and multiplication defined by $\bs{t}^{\bs{a}} \cdot \bs{t}^{\bs{b}} = \bs{t}^{\bs{a}+\bs{b}}$. It   is a subring of the ring of Laurent polynomials $\KK[\bs{t}^{\pm 1}]=\KK[t_1^{\pm 1},\ldots,t_d^{\pm 1}]$. % and $\KK[S]$ inherits a natural grading from any linear functional that is nonnegative on $S$. 

We can relate $R$ to  $\KK[S]$ via the $\KK$-algebra homomorphism
$$\varphi_{\mathcal{A}} : R \longrightarrow \KK[S], \hspace{1.5em} x_{\bs{\alpha}} \longmapsto \bs{t}^{\bs{{\alpha}}}.$$
We will take a more general look at this setting at the beginning of the next section. \\

\textit{Weight functions, term orders and initial ideals.} Let $P\subseteq \RR^d$ be a lattice polytope with set of lattice points $\mathcal{A}=\{\bs{\alpha}_1,\ldots,\bs{\alpha}_n\}=P\cap \ZZ^d$. 
%Let $\bs{x}^{\bs{b}} \coloneqq x_1^{b_1}\cdots x_n^{b_n} \in R$ be a monomial and associate a lattice point $\bs{a}_i \in P$ with $x_i \in R$.

A \emph{term order} $\prec$ on $R_P\coloneqq\KK[x_{\bs{\alpha}}:\bs{\alpha}\in \mathcal{A}]$ is a total order on the monomials in $R_P$ with (1) $1 \prec \bs{x}^{\bs{b}}$ for all non-unit monomials in $R_P$  and (2)  $\bs{x}^{\bs{b}+\bs{d}} \prec \bs{x}^{\bs{c}+\bs{d}}$ if $\bs{x}^{\bs{b}} \prec \bs{x}^{\bs{c}}$ and $\bs{d}\in \mathbb{N}^n$. 
For a polynomial $f \in R_P$, we define its \emph{initial monomial} $\initial_\prec(f)$ as the $\prec$-maximal monomial occurring in $f$ with non-zero coefficient.  

Let $I \subseteq R_P$ be any ideal. The \emph{initial ideal} with respect to the weighted term order $\prec_\omega$ is the monomial ideal
$$\initial_{\prec_\omega}(I) \coloneqq \langle \initial_{\prec_\omega}(f)~:~f\in I\rangle.$$

It is well-known that for every term order $\prec$ on $R_P$, there exists a weight function $\omega : \mc{A} \to \RR$ such that
\begin{align*}
    \bs{x}^{\bs{b}} \prec_\omega \bs{x}^{\bs{c}} \; \; \text{ if and only if } 
        \omega(\bs{x}^{\bs{b}}) < \omega(\bs{x}^{\bs{c}}), 
\end{align*}
where 
$$\omega(\bs{x}^{\bs{b}}) \coloneqq \sum_{i=1}^n b_i \omega(\bs{\alpha}_i)$$
is the so-called \emph{$\omega$-weight} of a monomial $\bs{x}^{\bs{b}}=\prod_{i=1}^nx_{\bs{\alpha}_i}^{b_i}\in R_P$.
We will write $\prec_{\omega}$ for the term order induced by $\omega$. 

The following well-known result, which is true for any homogeneous ideal in a polynomial ring and can be found, among others, in \cite{CoxLittleOShea2020}, is used in the upcoming sections:

\begin{corollary}\label{cor:hilb-series-hom-ideal}
    Suppose $I \subset R_P$ is a homogeneous ideal and $\prec$ is any term order. Then
    $$\Hilb_{R_P/I}(z) = \Hilb_{R_P/\initial_\prec(I)}(z).$$
\end{corollary} 

\vspace{1 em}

%%%--------------------------------------------------------------------------------%%%
%%%                           S E C T I O N     3                                  %%%
%%%--------------------------------------------------------------------------------%%%
\section{Ehrhart theory via commutative algebra}\label{sect:EhrhartComAlg} In this section we present an algebraic approach that relates (boundary) $h^\ast$-polynomials of lattice polytopes to regular (unimodular) triangulations (of the boundary) via the theory of initial ideals and Gr\"obner degenerations. The key ingredient will be Sturmfels' correspondence (see \cite[Theorem 8.3, Corollary 8.9]{Sturmfels1996}), which can be extended to triangulations of the boundary complex of $P$ (see \Cref{th:sturmfels-bd-corr}). Both these correspondences allow us to prove \Cref{th:betke-mcmullen} and \Cref{th:stapledon} in the special case where $P$, and $\partial P$ respectively, admit regular unimodular triangulations.

%%------------------------------ Subsection ---------------------------------------%
\subsection{Gr\"obner degenerations and triangulations} 
To keep this article self-contained, we start by recalling the famous correspondence between (squarefree) Gr\"obner degenerations of toric ideals and (unimodular) regular triangulations (see \cite[Sections 4 and 8]{Sturmfels1996} for more details).
%Given a regular triangulation $\mathcal{T}$ of $P$, recall from \Cref{sect:Prelim} that there exists a generic $\omega$ inducing $\mathcal{T}$ through lifting. 
Let $\mathcal{A} = \{\bs{\alpha}_1,\dots,\bs{\alpha}_n\}$ be a set of lattice points in $\ZZ^d$. Using the notation from \Cref{sect:Prelim}, we let $S=\mathbb{N}\{\mathcal{A}\}$ and $R=\KK[x_{\bs{\alpha}}:\bs{\alpha}\in \mathcal{A}]$. The kernel of the $\mathbb{K}$-algebra homomorphism 
\[
\varphi_{\mathcal{A}}:R \longrightarrow \KK[S], \hspace{1.5 em} x_{\bs{\alpha}} \longmapsto \bs{t}^{\bs{\alpha}}
\]
is called the \emph{toric ideal} of $\mathcal{A}$, denoted by $I_{\mathcal{A}}$. It is spanned as a vector-space by the set of binomials 
\[
\{ \bs{x}^{\bs{u}} - \bs{x}^{\bs{v}} : \bs{u},\bs{v} \in \NN^n,\; u_1 \bs{\alpha}_1 +\cdots+u_n \bs{\alpha}_n = v_1 \bs{\alpha}_1 +\cdots+v_n \bs{\alpha}_n\}.
\]

Let $I \subset R$ be any ideal and let $\prec$ be a term order. The \emph{initial complex} of $I$ is defined as the simplicial complex $\Delta_\prec(I)$ whose Stanley-Reisner ideal is equal to the radical of $\initial_\prec(I)$. This leads to Sturmfels' correspondence (see Theorem 8.3 and Corollary 8.9 in \cite{Sturmfels1996}):

\begin{theorem}\label{th:sturmfels-corr}
    The regular triangulations of $\conv(\mathcal{A})$ are the initial complexes of the toric ideal $I_{\mathcal{A}}$. More precisely, if $\omega \in \RR^{|\mc{A}|}$ represents $\prec$ for $I_{\mathcal{A}}$, then $\Delta_\prec(I_{\mathcal{A}}) = \Delta_{\omega}$.
\end{theorem}

On the level of ideals this means that the Stanley-Reisner ideal of the regular triangulation induced by $\omega\in \RR^{|\mc{A}|}$ equals the radical of the initial ideal of $I_{\mathcal{A}}$ with respect to the term order $\prec$ induced by $\omega \in \RR^{|\mc{A}|}$. Hence, their Hilbert series also coincide.

\begin{corollary}\label{cor:sturmfels-unimod}
    The initial ideal $\initial_\prec(I_{\mathcal{A}})$ is squarefree (i.e., it is equal to its radical ideal) if and only if the corresponding regular triangulation $\Delta_\prec$ of $\conv(\mathcal{A})$ is unimodular.
\end{corollary}

Now we come back to our polytope setting and let $\mathcal{A} = \{\bs{\alpha}_1,\dots,\bs{\alpha}_n\} = P \cap \ZZ^d$ for a lattice polytope $P\subseteq \RR^d$ (see \Cref{sect:Prelim}). We consider the semigroup $S_P \coloneqq \cone(P) \cap \ZZ^{d+1}$ which is finitely generated (see Gordan's Lemma \cite[Theorem\ 7.16]{MillerSturmfels2005}). The \emph{Ehrhart ring} of $P$ is the semigroup algebra 
\[
\KK[P] \coloneqq \KK[S_P]= \KK[\bs{t}^{(\bs{b},m)}\; :\; (\bs{b},m) \in S_P],
\]
which carries a natural $\NN$-grading given by $\deg\bs{t}^{(\bs{b},m)} = m$. Consequently, for each $m \geq 0$, the graded component $(\KK[P])_m$ has dimension $\dim_\KK(\KK[P])_m = |mP \cap \ZZ^d| = L_P(m).$
It immediately follows, without any further assumptions on $P$, that 
\begin{align}\label{eq:hilb-ehrhart-P}
    \Hilb_{\KK[P]}(z) = \sum_{m \geq 0} \dim_\KK(\KK[P])_m z^m = \sum_{m \geq 0} L_P(m) z^m = \Ehr_P(z).
\end{align}

In the case of $P$ admitting a regular unimodular triangulation $\mathcal{T}$, we are already able to prove a special case of \Cref{th:betke-mcmullen}.

\begin{corollary}\label{cor:betke-mcmullen-spec}
    If $P \subset \RR^d$ is a lattice $d$-polytope that admits a regular unimodular triangulation $\mathcal{T}$, then $h^\ast_P(z) = h_\mathcal{T}(z)$.
\end{corollary}

\begin{proof}
    To simplify notation, we set $R_P\coloneqq  \KK[x_{\bs{\alpha}}:\bs{\alpha}\in \mc{A}]$. Consider the homomorphism of graded $\KK$-algebras
    \begin{align}\label{eq:homomorphism}
        \varphi_P: R_P \longrightarrow \KK[P], \hspace{1.5 em} x_{\bs{\alpha}} \longmapsto \bs{t}^{(\bs{\alpha},1)}.
    \end{align}
    Its kernel $\ker(\varphi_P) =: I_P$ is the toric ideal associated with $\{(\bs{\alpha},1) \; : \; \bs{\alpha} \in \mathcal{A}\}$. Since $\mathcal{T}$ is unimodular, $P$ is IDP. Hence, $\varphi_P$ is surjective and
    \begin{align}\label{eq:isomorphic-P}
        R_P/I_P \cong \KK[P]
    \end{align}
    by the first isomorphism theorem. Note that $I_P$ is homogeneous and encodes the algebraic relations among the elements in $\mathcal{A}$ through binomials $\bs{x}^{\bs{u}} - \bs{x}^{\bs{v}} \in R_P$. 

    By \Cref{cor:hilb-series-hom-ideal},
    \[
    \Hilb_{R/I_P}(z) = \Hilb_{R/\initial_{\prec_w}(I_P)}(z),
    \]
    where $\prec_\omega$ is the weighted term order with respect to $\omega : \mc{A} \to \RR$ corresponding to $\mathcal{T}$. 
    
    Applying \Cref{th:sturmfels-corr}, one obtains that the radical of $\initial_{\prec_\omega}(I_P)$ equals the Stanley-Reisner ideal $I_\mathcal{T}$. In particular, by \Cref{cor:sturmfels-unimod}, unimodularity of $\mathcal{T}$ implies that $\initial_{\prec_\omega}(I_P)$ is squarefree. Hence, 
    \[
    \Hilb_{R/\initial_{\prec_\omega}(I_P)}(z) = \Hilb_{R/I_\mathcal{T}}(z) = \frac{h_\mathcal{T}(z)}{(1-z)^{d+1}},
    \]
    which, together with (\ref{eq:hilb-ehrhart-P}) and (\ref{eq:isomorphic-P}), completes the proof.
\end{proof}

%%%------------------------------ Subsection ----------------------------------------%
\subsection{Toric ideals of boundary complexes of lattice polytopes} For the boundary analogue of \Cref{cor:betke-mcmullen-spec}, we again start with no assumptions on $P$ besides being a lattice $d$-polytope in $\RR^d$. As in the proof of \Cref{cor:betke-mcmullen-spec}, we set $R_P= \KK[x_{\bs{\alpha}}:\bs{\alpha}\in \mc{A}]$ and write $\varphi_P$ for the homomorphism in \eqref{eq:homomorphism}. Moreover, throughout this section, let $n=|P\cap \ZZ^d|$. Translating the setting of the previous subsection, let $S_P^\circ \coloneqq (\cone(P)^\circ \cap \ZZ^{d+1})$ be the monoid  of all lattice points in the interior $\cone(P)^\circ$ of  $\cone(P)$. 
In order to characterize monomials in $\KK[P]$ corresponding to lattice points on the boundary, we consider the following homogeneous \emph{interior} semigroup ideal
\[
J_{P^\circ} \coloneqq \langle \bs{t}^{(\bs{b},m)} \; : \; (\bs{b},m) \in S_P^\circ\rangle \subseteq \KK[P]
\]
together with the composition of homomorphisms
\[
R_P \xrightarrow[]{\varphi_P} \KK[P] \xrightarrow[]{\pi} \KK[P]/J_{P^\circ}.
\]
Note that $\KK[P]/J_{P^\circ}$ is not a semigroup algebra, but the monoid algebra of the Rees quotient monoid $S_P/S_P^\circ$. Since $(\KK[P])_m$ has basis $\{\bs{t}^{(\bs{b},m)}\; : \; \bs{b} \in mP \cap \ZZ^d\}$ and $(J_{P^\circ})_m$ has basis $\{\bs{t}^{(\bs{b},m)}\; :\; \bs{b} \in (mP)^\circ \cap \ZZ^d\}$, it follows that
\[
\dim_\KK(\KK[P]/J_{P^\circ})_m = |mP \cap \ZZ^d| - |(mP)^\circ\cap\ZZ^d| = |\partial(mP) \cap \ZZ^d| = L_{\partial P}(m).
\]
Again, we immediately obtain
\begin{align}\label{eq:hilb-ehrhart-bdP}
    \Hilb_{\KK[P]/J_{P^\circ}}(z) = \Ehr_{\partial P}(z).
\end{align}

Our next objective is to compute the  kernel of $\pi \circ \varphi_P$ and show that it has certain properties, allowing us to use the quotient ring of $R_P$ modulo this ideal as in the non-boundary case. Define $\varphi_P^{-1}(J_{P^\circ}) \subseteq R_P$ to be the pull-back ideal of $J_{P^\circ}$. We need the following definition throughout the rest of the section: for a monomial $\bs{x}^{\bs{c}} \in R_P$ with $\bs{c} \in \NN^n$ define $\gamma(\bs{c})\coloneqq \sum_{\bs{\alpha}\in \mc{A}} c_{\bs{\alpha}}(\bs{\alpha},1)$ to be the exponent of $\varphi_P(\bs{x}^{\bs{c}})$ (as a monomial in $\bs{t}$).

\begin{proposition}\label{prop:general-prop}
    The ideal $\varphi_P^{-1}(J_{P^\circ})$ has the following properties:
    \begin{enumerate}
        \item[(i)] $I_P \subseteq \varphi_P^{-1}(J_{P^\circ}).$
        \item[(ii)] For $\bs{c} \in \NN^n$ we have: $\bs{x}^{\bs{c}} \in \varphi^{-1}_P(J_{P^\circ}) \Longleftrightarrow \gamma(\bs{c}) \in S_P^\circ$.
        \item[(iii)] $\varphi_P^{-1}(J_{P^\circ})$ is homogeneous and $\varphi_P^{-1}(J_{P^\circ}) = I_P + M_P$ with $M_P = \langle \bs{x}^{\bs{c}}: \gamma(\bs{c}) \in S_P^\circ\rangle$.
    \end{enumerate}
\end{proposition}

\begin{proof}
    We first prove (i). Take any $f \in I_P = \ker(\varphi_P)$. Then $\varphi_P(f) = 0 \in \KK[P]$. Since $J_{P^\circ}$ is an ideal of $\KK[P]$, it contains $0$. Hence, $\varphi_P(f) \in J_{P^\circ}$, implying $f \in \varphi_P^{-1}(J_{P^\circ})$. \\

    Next, we show (ii). $\Rightarrow:$ Assume $\bs{x}^{\bs{c}} \in \varphi^{-1}_P(J_{P^\circ})$. Then $\varphi_P(\bs{x}^{\bs{c}})= \bs{t}^{\gamma(\bs{c})} \in J_{P^\circ}$. Since $J_{P^\circ}$ is generated by $\{\bs{t}^{\bs{u}}:\bs{u} \in S_P^\circ\}$, it follows that 
    \[
    \gamma(\bs{c}) = \bs{s} + \bs{s}^\circ
    \]
    for some $\bs{s} \in S_P, \bs{s}^\circ \in S_P^\circ$ (see \cite{MillerSturmfels2005}).

    Since for any rational polyhedral $(d+1)$-dimensional cone $C \subset \RR^{d+1}$, one has
    \[
    C^\circ + C \subseteq C^\circ,
    \]
    we conclude that $\gamma(\bs{c}) = \bs{s}+\bs{s}^\circ \in \cone(P)^\circ$. Since $\gamma(\bs{c})$ is integral, $\gamma(\bs{c}) \in S_P^\circ$.

    \noindent $\Leftarrow:$ Assume $\gamma(\bs{c}) \in S_P^\circ$. Then $\varphi_P(\bs{x}^{\bs{c}})=\bs{t}^{\gamma(\bs{c})}\in J_{P^\circ}$, because $\bs{t}^{\gamma(\bs{c})}$ is one of the generators of $J_{P^\circ}$. Hence, $\bs{x}^{\bs{c}} \in \varphi^{-1}_P(J_{P^\circ})$.\\

    For (iii), we begin with the decomposition of $\varphi_P^{-1}(J_{P^\circ})$:

    The inclusion $\supseteq$ directly follows from (i) and the observation that $\bs{x}^{\bs{c}}\in\varphi_P^{-1}(J_{P^\circ})$  for every generator $\bs{x}^{\bs{c}}$ of $M_P$.

    The inclusion   $\subseteq$ is a bit more involved. Let $f \in \varphi^{-1}_P(J_{P^\circ})$. We write $f$ as a finite $\KK$-linear combination of monomials and apply $\varphi_P$:
    \begin{align}\label{eq:decomposition-image}
        \varphi_P(f) = \sum_{j=1}^r \lambda_j \bs{t}^{\gamma(\bs{c}^{(j)})} \in J_{P^\circ}, \hspace{1.5 em} \lambda_j \in \KK, \; \bs{c}^{(j)} \in \NN^n.
    \end{align}
    Grouping together terms with equal exponents, we can rewrite (\ref{eq:decomposition-image}) as
    \[
    \varphi_P(f) = \sum_{\bs{u}\in S_P} \alpha_{\bs{u}} \bs{t}^{\bs{u}}
    \]
    with only finitely many nonzero $\alpha_{\bs{u}} \in \KK$.  Split $f$ into two parts: $f_{\circ}$ is the sum of those terms $\lambda_j\bs{x}^{\bs{c}^{(j)}}$ with $\gamma(\bs{c}^{(j)})\in S_P^\circ$ and $f_\partial$ denotes the sum of those terms $\lambda_j\bs{x}^{\bs{c}^{(j)}}$ with $\gamma(\bs{c}^{(j)})\notin S_P^\circ$.
    By definition, $f_\circ \in M_P$. It remains to show $f_\partial \in I_P$. Applying $\varphi_P$ we get:
    \[
    \varphi_P(f_\partial) = \sum_{\gamma(\bs{c}^{(j)})\notin S_P^\circ} \lambda_j \bs{t}^{\gamma(\bs{c}^{(j)})}.
    \]

    By what we observed above, all coefficients $\alpha_{\bs{u}}$ of monomials $\bs{t}^{\bs{u}}$ with $\bs{u} \notin S_P^\circ$ must be zero. Since $\varphi_P(f_\circ) \in J_{P^\circ}$ contributes only monomials with exponents in $S_P^\circ$, it cannot affect those coefficients. Hence, every coefficient of $\varphi_P(f_\partial)$ must be zero, i.e.\ $\varphi_P(f_\partial) = 0$. Thus $f_\partial \in \ker(\varphi_P) = I_P$. Therefore $f = f_\circ + f_\partial \in M_P + I_P$, proving the inclusion $\subseteq$. \\
    
    Finally, we complete the proof of (iii): Since $\varphi_P$ is a graded ring homomorphism and $J_{P^\circ}$ is homogeneous, the homogeneity of $\varphi_P^{-1}(J_{P^\circ})$ follows from the general fact that pullbacks of homogeneous ideals under graded homomorphisms remain homogeneous.
\end{proof}

To move from $\KK[P]/J_{P^\circ}$ to $R_P/\varphi^{-1}_P(J_{P^\circ})$ requires $\pi \circ \varphi_P$ to be surjective. This holds, as the next lemma will show, if $\partial P$ admits a unimodular triangulation $\Delta$, enabling us to follow up on (\ref{eq:hilb-ehrhart-bdP}) in the following way:

\begin{lemma}\label{lem:unimod-prop}
    Let $P\subseteq \RR^d$ be a lattice $d$-polytope. If $\partial P$ admits a unimodular triangulation $\Delta$, then, $\pi \circ \varphi_P$ is surjective (onto the boundary quotient $\KK[P]/J_{P^\circ}$) with $\varphi_P^{-1}(J_{P^\circ}) = \ker(\pi \circ \varphi_P)$ and $R_P/\varphi_P^{-1}(J_{P^\circ}) \cong \KK[P]/J_{P^\circ}$.
\end{lemma}

\begin{proof}
    The equality $\varphi_P^{-1}(J_{P^\circ}) = \ker(\pi \circ \varphi_P)$ holds by definition. For surjectivity, fix $m\geq 0$. Recall that by construction, $(\KK[P]/J_{P^\circ})_m$ is spanned by the residue classes of monomials $\bs{t}^{(\bs{b},m)}$ with $\bs{b} \in \partial(mP) \cap \ZZ^d$. Let $\bs{b} \in \partial(mP)\cap \ZZ^d$. Then $\bs{b} \in mF$ for a (not necessarily unique) facet $F$ of $P$. Since $\Delta$ is unimodular, its restriction $\Delta_{|F}$ is a unimodular triangulation of the lattice $(d-1)$-polytope $F$. Hence, $F$ is IDP in its affine lattice, i.e., there exist $\bs{y}_1, \dots,\bs{y}_m \in mF \cap \ZZ^d \subseteq \partial P \cap \ZZ^d$ with
    \[
    \bs{b} = \bs{y}_1+\cdots+\bs{y}_m.
    \]
    Therefore, $\bs{t}^{(\bs{b},m)} = \bs{t}^{(\bs{y}_1,1)} \cdots \bs{t}^{(\bs{y}_m,1)}$ in $\KK[P]$, so the class of $\bs{t}^{(\bs{b},m)}$ in $\KK[P]/J_{P^\circ}$ lies in the image of $\pi \circ \varphi_P$. Since these classes span $(\KK[S]/J_{P^\circ})_m$, the map $\pi \circ \varphi_P$ is surjective in each degree, hence surjective. The last part then follows by the first isomorphism theorem.
\end{proof}

If $\partial P$ admits a unimodular triangulation, we can conclude from \Cref{lem:unimod-prop} the following identity of Hilbert series:
\begin{align}\label{eq:hilb-series-hom-ideal-P}
    \Hilb_{\KK[P]/J_{P^\circ}}(z) = \Hilb_{R_P/\varphi^{-1}_P(J_{P^\circ})}(z).
\end{align}
We call $I_{\partial P}\coloneqq\varphi^{-1}_P(J_{P^\circ})$ the \emph{toric boundary ideal} of  the point configuration $P$. Note that such an ideal can be associated to any point configuration  $\mathcal{A}$ and its boundary configuration. We also want to note that the name \emph{toric} might be misleading since $I_{\partial P}$ is not a toric ideal in the usual sense but rather consists of a toric and a monomial component. 

%%%------------------------------ Subsection ----------------------------------------%
\subsection{The arithmetic and face structure of $\partial P$} Before we can prove the main result of our work, we have to show that a regular triangulation $\Delta$ of $\partial P$ can be extended to a regular triangulation $\mc{T}$ of $P$, allowing us to work with initial ideals with respect to a weighted term order induced by this extension. If the boundary triangulation $\Delta$ is induced by a weight function $\omega_\partial : \mc{B} \to \RR$ with $\mc{B} \coloneqq \partial P \cap \ZZ^d$, $\mc{B} \subseteq \mc{A}=P\cap \ZZ^d$, the following two results are essential.%
%The proofs are tedious and the reader might want to skip them on the first read-through and come back to them later.

\begin{lemma}\label{lem:subdivision-facets}\cite[Lemma 2.3.15]{DeLoeraRambauSantos2010}
    Let $P \subset \RR^d$ be a $d$-polytope, $\mathcal{A} \subset P$ be a finite point configuration containing all vertices of $P$  and $\omega:\mathcal{A}\to \RR$ be any weight function. Let $\mc{S}_\omega(P)$ be the regular subdivision of $P$ induced by $\omega$ via the lifting construction. Then, for every face $F$ of $P$, the induced subdivision on $F$ is exactly the regular subdivision $\mc{S}_{\omega_{|\mathcal{A}\cap F}}(F)$ induced by $\omega_{|\mathcal{A}\cap F}$. 
\end{lemma}

%\begin{proof}
 %   Consider the lifted point set $\Tilde{\mathcal{A}} \coloneqq \{(\bs{a},\omega(\bs{a})):\bs{a}\in A\} \subset \RR^{d+1}$ and set $Q\coloneqq \text{conv}(\Tilde{\mathcal{A}})$. The regular subdivision $\mc{S}_\omega(P)$ consists of the projections of the lower faces of $Q$ to $\RR^d$. Now fix a face $F$ of $P$ and choose a supporting hyperplane $H \subset \RR^d$. Then the lifted points with base in $F$ all lie in the affine subspace $H \times \RR \subset \RR^{d+1}$. Moreover,
  %  \[
    %\text{conv}\{(\bs{a},\omega(\bs{a})):\bs{a}\in \mathcal{A} \cap F\} = Q \cap (H \times \RR),
   % \]
   %since  intersecting $\text{conv}(\Tilde{\mathcal{A}})$ with the vertical affine subspace over $H$ removes exactly those vertices whose base point is not in $H$.
%
    %A face $G$ of $Q \cap (H \times \RR)$ is a lower face if and only if it is the intersection of a lower face of $Q$  with $H \times \RR$  Therefore, if one first projects lower faces of $Q$ to $\RR^d$ and  then restricts to $F$, one gets precisely the projections of lower faces of $\text{conv}\{\bs{a},\omega(\bs{a})):\bs{a}\in \mathcal{A} \cap F\}$. This is exactly the regular subdivision induced on $F$ by the restricted weight function $\omega_{|A \cap F}$. 
%\end{proof}

Consequently, if $\omega:\mc{A}\to \RR$ extends $\omega_\partial:\mc{B}\to\RR$, then for every facet $F$ of $P$,
\begin{align}
    \mc{S}_\omega(P)_{|F} = \mc{S}_{\omega_\partial}(F)=\Delta_{|F},
\end{align}
i.e., boundary regularity is preserved by any extension. 
The second lemma is constructive in extending $\omega_\partial$ to $\omega$ such that the induced triangulation $\mc{T}$ of $P$ is regular, while $\Delta$ remains unchanged under $\omega$.

\begin{lemma}\label{lem:extension-weight-fct}
    Let $P \subset \RR^d$ be a  lattice  $d$-polytope with $\mathcal{A}=P\cap \ZZ^d$ and $\mathcal{B}=\partial P\cap \ZZ^d$. Let $\omega_\partial :\mc{B} \to \RR$ be a weight function and let $\Delta$ be the induced regular triangulation  of $\partial P$. Then there exists an extension $\omega:\mc{A}\to\RR$ such that the boundary complex of the regular subdivision $\mc{S}_\omega(P)$ of $P$ equals $\Delta$, i.e.,  $\partial \mc{S}_\omega(P) = \Delta$. 
\end{lemma}
Indeed, it is true that such an extension can be chosen from an open dense set in $\{\omega_\partial\}\times \RR^{\mathcal{A} \setminus \mathcal{B}}$. An extension as in the lemma will be referred to as $\omega_{\partial}$-\emph{compatible} extension. 
%Though we assume that this lemma is well-known, we could not find it in the literature and therefore include a proof.
Though we assume the statement to be well-known, we include an easy proof since we could not find it in the literature.
\begin{proof}
    Start with any extension $\omega^{(0)}:\mc{A} \to \RR$ for $\omega_\partial$. For instance, set $\omega^{(0)}(\bs{a}) = \omega_\partial(\bs{a})$ for $\bs{a} \in \mc{B}$ and $\omega^{(0)}(\bs{a}) = 0$ for $\bs{a} \in \mc{A}\setminus \mc{B}$. 
    
    $\omega^{(0)}$ induces a regular subdivision $\mc{S}_{\omega^{(0)}}$ of $P$, by definition and, by \Cref{lem:subdivision-facets}, the restriction of $\mc{S}_{\omega^{(0)}}$ to any facet $F$ of $P$ equals $\Delta_{|F}$. In particular, the subdivision induced on $\partial P$ is $\Delta$ (since $\Delta$ is exactly the collection of these facet restrictions glued compatibly).
    The claim follows directly, if $\mc{S}_{\omega^{(0)}}$ is already a triangulation. If not, we can refine $\mc{S}_{\omega^{(0)}}$ to a regular triangulation (see e.g., \cite[Corollary 2.3.18]{DeLoeraRambauSantos2010}).
\end{proof}

Now we can continue with (\ref{eq:hilb-series-hom-ideal-P}) in order to prove the boundary analogue of \Cref{th:sturmfels-corr} in our setting of lattice $d$-polytopes. Let $\Delta$ be the regular triangulation of $\partial P$ induced by the weight function $\omega_\partial :\mc{B} \to \RR$. Let $\omega:\mc{A}\to \RR$ be a global weight function  extending $\Delta$ to a regular triangulation $\mc{T}$ of $P$ with $\partial \mc{T} = \Delta$ (see \Cref{lem:extension-weight-fct}) and let $\prec_\omega$ be the corresponding term order. We consider 
\[
\initial_{\prec_{\omega}}(I_{\partial P})= \initial_{\prec_\omega}(I_P + M_P),
\]
where the identity follows by \Cref{prop:general-prop} (iii). Since monomials are fixed by taking initial ideals and initial ideals commute with sums when one summand is monomial (see \cite{CoxLittleOShea2020,MillerSturmfels2005,Sturmfels1996}), we get
\begin{align}\label{eq:decomposition-J}
    \initial_{\prec_{\omega}}(I_{\partial P}) = \initial_{\prec_\omega}(I_P) + M_P.
\end{align}

The next lemma tells us what restricting to a facet $F$ of $P$ means for the initial ideal $\initial_{\prec_{\omega}}(I_{\partial P})$.

\begin{lemma}\label{lem:restriction-J-facets}
    Let $\Delta$ be the regular triangulation of $\partial P$ induced by $\omega_\partial$ and let $\omega$ be a $\omega_\partial$-compatible extension. Further, for a facet $F$ of $P$, let $\mc{A}_F \coloneqq F \cap \ZZ^d$, $R_F \coloneqq \KK[x_{\bs{\alpha}}:\bs{\alpha} \in \mc{A}_F] \subseteq R_P$ and let 
    \[
    \varphi_F:R_F \to \KK[S_F], \hspace{1.5 em} x_{\bs{\alpha}}\mapsto \bs{t}^{(\bs{\alpha}_i,1)}
    \]
    be the semigroup map associated to $F$, where $S_F\coloneqq \cone(F) \cap \ZZ^{d+1}$. Let $I_F\coloneqq\ker(\varphi_F) \subseteq R_F$ denote the toric ideal of $\mc{A}_F$ and let $\omega_F$ be the restriction of $\omega$ to the lattice points in $\mathcal{A}_F$. %variables of $R_F$. 
    Then 
    \[
    \initial_{\prec_{\omega}}(I_{\partial P}) \cap R_F = \initial_{\prec_{\omega_F}}(I_F).
    \]
\end{lemma}

\begin{proof}
    By (\ref{eq:decomposition-J}) it is enough to understand the intersections of the two summands of $\initial_{\prec_{\omega}}(I_{\partial P})$ with $R_F$. We proceed in several steps.\\

    \noindent\textit{Step 1:} $M_P \cap R_F = 0$. Let $\bs{x}^{\bs{c}} \in R_F$ be a monomial. Then all variables appearing in $\bs{x}^{\bs{c}}$ correspond to lattice points of $F$. Hence
    \[
    \gamma(\bs{c}) = \sum_{\bs{\alpha}\in \mathcal{A}_F} c_{(\bs{\alpha},1)} \in\cone(F).
    \]
    Since $\cone(F)$ is a proper face of $\cone(P)$, we have  $\gamma(\bs{c}) \notin S_P^\circ$ and hence $\bs{x}^{\bs{c}} \notin M_P$. Consequently, $\initial_{\prec_{\omega}}(I_{\partial P}) \cap R_F = \initial_{\prec_\omega}(I_P) \cap R_F$. \\

    \noindent\textit{Step 2:} $I_P \cap R_F = I_F$. The inclusion $I_F \subseteq I_P \cap R_F$ is immediate since $\varphi_P|_{R_F}$ agrees with $\varphi_F$. For the reverse inclusion, let $f \in I_P \cap R_F$. Since $f \in R_F$, every monomial of $f$ is supported on variables corresponding to lattice points of $F$. This implies $\varphi_F(f) = \varphi_P(f)=0$ and hence, $f \in I_F$.\\

    \noindent\textit{Step 3:} $\initial_{\prec_\omega}(I_P) \cap R_F = \initial_{\prec_{\omega_F}}(I_F)$. To prove $\supseteq$, let $p \in \initial_{\prec_{\omega_F}}(I_F)$. Then there exists $f \in I_F$ with $\initial_{\prec_{\omega_F}}(f) = p_0$ such that  $p_0$ divides $p$. Since $I_F \subseteq I_P$ and all monomials of $f$ lie in $R_F$, we have 
    \[
    \initial_{\prec_\omega}(f) = \initial_{\prec_{\omega_F}}(f) = p_0 \in \initial_{\prec_\omega}(I_P)
    \]
    and, since $\initial_{\prec_\omega}(I_P)$ is a monomial ideal, it follows that $p \in \initial_{\prec_\omega}(I_P)$. Clearly, $p \in R_F$, so $p \in \initial_{\prec_\omega}(I_P) \cap R_F$.
    
    To prove the inclusion $\subseteq$, let $\bs{x}^{\bs{c}} \in \initial_{\prec_\omega}(I_P) \cap R_F$ be a monomial. By definition of the initial ideal, there exists $f \in I_P$ be such that  $\initial_{\prec_{\omega}}(f) = \bs{x}^{\bs{c}}$. Choose $f$ minimal with this property, i.e., there is no polynomial $g\in I_P$ whose support is strictly contained in the support of $f$ such that $\initial_{\prec_{\omega}}(g)=\bs{x}^{\bs{c}}$. 
    %Since $ \bs{x}^{\bs{c}}\in R_F$, every variable appearing in $\bs{x}^{\bs{c}}$ corresponds to a lattice point in $F$.  
    We show that  every monomial in the support of $f$ is supported on $\mathcal{A}_F$. Note that this is true for  $\bs{x}^{\bs{c}}$, by assumption. Write 
    \[
    f = \bs{x}^{\bs{c}}+\sum_{j=1}^r \lambda_j \bs{x}^{\bs{c}^{(j)}},
    \]
    with  $\bs{c}^{(j)} \in \NN^n$. Since $f \in I_P$, we have $\varphi_P(f)=0$ and, due to minimality of $f$, it follows that $\gamma(\bs{c}^{(j)}) = \gamma(\bs{c})$ for all $1\leq j\leq r$. Since $\bs{x}^{\bs{c}} \in R_F$, we conclude that $\gamma(\bs{c}) \in \cone(F)$.% But $\gamma(\bs{c})$ is exactly the common semigroup degree $\bs{v}$, so $\bs{v} \in \cone(F)$. 

    Now we choose a supporting affine functional $\ell:\RR^d \to \RR$ for the facet $F$, normalized so that $\ell(\bs{x}) \leq 1$ for all $\bs{x} \in P$ and $\ell(\bs{x}) = 1$ for all $\bs{x} \in F$. Define the associated linear functional on $\RR^{d+1}$:
    \[
    \Tilde{\ell}(\bs{x},s) = \ell(\bs{x}) - s.
    \]
    Then $\Tilde{\ell}(\bs{\alpha},1) \leq 0$ for every $\bs{\alpha} \in \mc{A}$ and $\Tilde{\ell}(\bs{\alpha},1) = 0$ if and only if $\bs{\alpha} \in F$. Moreover, $\cone(F) = \cone(P) \cap \{\Tilde{\ell} = 0\}$. 
    In particular, since $\gamma(\bs{c}^{(j)})=\gamma(\bs{c})\in \cone(F)$, it follows that
    
    \[
    0=\Tilde{\ell}(\gamma(\bs{c}^{(j)})) = \sum_{\bs{\alpha}\in\mathcal{A}} c_{\bs{\alpha}}^{(j)} \Tilde{\ell}(\bs{\alpha},1),
    \]
    %Take any monomial $\bs{x}^{\bs{c}^{(j)}}$ appearing in $f$. Since its semigroup degree is $\bs{v} \in \cone(F)$, we have
    %\[
   % \gamma(\bs{c}^{(j)}) = \bs{v} \hspace{1em} \Longrightarrow \hspace{1em} \Tilde{\ell}(\gamma(\bs{c}^{(j)}) = 0.
   % \]
   Each summand in the last sum satisfies $c_{\bs{\alpha}}^{(j)} \cdot \Tilde{\ell}(\bs{\alpha},1) \leq 0$, because $c_{\bs{\alpha}}^{(j)} \geq 0$ and $\Tilde{\ell}(\bs{\alpha},1) \leq 0$. Using that the whole sum is zero, this implies that  every summand with $c_{\bs{\alpha}}^{(j)} > 0$ must satisfy $\Tilde{\ell}(\bs{\alpha},1) = 0$. Therefore, every variable appearing in $\bs{x}^{\bs{c}^{(j)}}$ corresponds to a lattice point $\bs{\alpha} \in F$. It follows that $\bs{x}^{\bs{c}^{(j)}}\in R_F$, for $1\leq j\leq r$ and hence $f \in R_F$. Since $f \in I_P$, \emph{Step 2} implies that  $f \in I_P \cap R_F = I_F$.

   Finally, since all monomials in $f$ lie in $R_F$, we conclude
   \[
   \bs{x}^{\bs{c}}= \initial_{\prec_\omega}(f)=\initial_{\prec_{\omega_F}}(f)\in  \initial_{\prec_{\omega_F}}(I_F),
   \]
   finishing the proof of the second inclusion.
\end{proof}

\noindent We are now prepared to prove our first  main result of this section.

\begin{theorem}\label{th:sturmfels-bd-corr}
    Let $P \subset \RR^d$ be a lattice $d$-polytope with $\mathcal{A}=P\cap \ZZ^d$ and $\mathcal{B}=\partial P\cap \ZZ^d$.   The regular triangulations of the boundary complex $\partial P$ are the initial complexes of $I_{\partial P}$. More precisely, if $\omega_\partial\in \RR^{|\mathcal{B}|}$ induces the regular triangulation $\Delta$ of $\partial P$, and $\omega\in \RR^{|\mathcal{A}|}$ is a $\omega_\partial$-compatible extension, then
    \[
    \sqrt{\initial_{\prec_\omega}(I_{\partial P})} = I_\Delta.
    \]
\end{theorem}

\begin{proof}
    Let $\mathcal{T}$ be the triangulation of $P$ induced by $\omega$. 
    %We need to prove that
    %\begin{align*}
     %   \sqrt{J} = I_{\partial \mc{T}} = I_\Delta.
    % \end{align*}
    By (\ref{eq:decomposition-J}), taking the radical of $\initial_{\prec_\omega}(I_{\partial P})$ gives
    \[
        \sqrt{\initial_{\prec_\omega}(I_{\partial P})}=\sqrt{\initial_{\prec_\omega}(I_P)+M_P} =  \sqrt{\sqrt{\initial_{\prec_\omega}(I_P)}+M_P},
    \]
    where the second identity is a basic fact for radicals. Using \Cref{th:sturmfels-corr}, the radical of $\initial_{\prec_\omega}(I_P)$ is equal to the Stanley-Reisner ideal of $\mc{T}$ and it thus remains to prove that 
    \begin{align}\label{eq:to-prove}
        \sqrt{I_\mc{T}+M_P} = I_\Delta.
    \end{align}
    Let $\sigma \in \mc{T}$ be a face. First, we prove the following claim.

    \noindent\textit{Claim:}  $\sigma\subseteq \partial P$ if and only if  $\bs{x}^\sigma \notin \sqrt{M_P}$.
    %\begin{enumerate}
     %   \item[(i)] if $\text{conv}\{\bs{a}_i:i\in \sigma\}\subseteq \partial P$, then $\bs{x}^\sigma \notin \sqrt{M_P}$
      %  \item[(ii)] if $\text{conv}\{\bs{a}_i:i\in \sigma\}\nsubseteq \partial P$, then $\bs{x}^\sigma \in \sqrt{M_P}.$
   % \end{enumerate}
    
    Set 
    %In order to characterize the interpretation of $\bs{x}^\sigma \in \sqrt{M_P}$, let 
    \[
    \bs{u}_\sigma\coloneqq \sum_{\bs{\alpha} \in \sigma} (\bs{\alpha},1) = \left( \sum_{\bs{\alpha}\in \sigma}\bs{\alpha},|\sigma| \right).
    \]
    Observe that 
    \[
    (\bs{x}^\sigma)^k \in M_P \Longleftrightarrow k\bs{u}_\sigma \in S_P^\circ \Longleftrightarrow \bs{u}_{\sigma} \in \cone(P)^\circ.
    \]
    
    Suppose first that $\sigma \subseteq \partial P$. Then $\sigma$ is contained in some (not necessarily unique) facet $F$ of $P$. Let $\ell$ and $\Tilde{\ell}$ be chosen as in the proof of \Cref{lem:restriction-J-facets}. Then $\Tilde{\ell} \leq 0$ on $\cone(P)$ with equality precisely on $\cone(F)$. Since $\Tilde{\ell}(\bs{u}_\sigma) = 0$, we have $\bs{u}_\sigma\in\partial( \cone(P))$ , i.e., $\bs{u}_\sigma \notin \cone(P)^\circ$ and hence $\bs{x}^\sigma \notin \sqrt{M_P}$.
 
    Conversely, suppose that $\sigma \nsubseteq \partial P$. Then for every facet functional $\ell_j$ on $P$, $\ell_j(\bs{\alpha}) \leq 1$ for all $\bs{\alpha} \in \sigma$ and for each $j$ there exists $\bs{\alpha} \in \sigma$ such that $\ell_j(\bs{\alpha}) < 1$. Hence
    $$\ell_j\left( \sum_{\bs{\alpha}\in \sigma}\bs{\alpha} \right) < |\sigma|.$$ 
    %Recall that if $P = \{\bs{x}:\ell_j(\bs{x})\leq 1, \; j=1,\ldots,r\}$ is given by its facets inequalities, then 
    %\[
    %\cone(P) = \{(\bs{x},m):m \geq 0, \ell_j(\bs{x})\leq m \text{ for all } j \}
    %\]
    %and 
    %\[
    %\reli(\cone(P)) = \{(\bs{x},m):m \geq 0, \ell_j(\bs{x})< m \text{ for all } j \}.
    %\]
    Thus, by linearity of $\ell_j$, $\bs{u}_\sigma$ satisfies all strict facet inequalities of $\cone(P)^\circ$, implying $\bs{u}_\sigma \in S_P^\circ$. Therefore $\bs{x}^\sigma \in \sqrt{M_P}$. The claim follows.

    Now we can connect $ \sqrt{\initial_{\prec_\omega}(I_{\partial P})}$ to the Stanley-Reisner ideal $I_\Delta$ of $\Delta$. Since both $\sqrt{I_\mc{T}+M_P}$ and $I_\Delta$ are squarefree monomial ideals, it suffices to compare the squarefree monomials they contain.

    Let $\sigma \subseteq \mathcal{A}$. If $\bs{x}^\sigma \in \sqrt{I_\mc{T}+M_P}$, then $\bs{x}^\sigma$ is divisible by some squarefree monomial generator of $\sqrt{I_\mc{T}+M_P}$. Equivalently, some subset $\sigma^\prime \subseteq\sigma$ satisfies either $\bs{x}^{\sigma^\prime} \in I_\mc{T}$ or $\bs{x}^{\sigma^\prime} \in \sqrt{M_P}$. In the first case,  $\sigma$ is not a face of $\mc{T}$, in the second case,  one obtains $\sigma \nsubseteq \partial P$ by the above claim. In either case $\sigma$ is not a face of $\partial \mc{T}$, so $\bs{x}^\sigma \in I_{\partial \mc{T}} = I_\Delta$. 

    Conversely, if $\sigma$ is not a face of $\partial \mc{T}$, then either $\sigma$ is not a face of $\mc{T}$ (so $\bs{x}^\sigma \in I_\mc{T}$) or $\sigma$ is a face of $\mc{T}$ not contained in $\partial P$, in which case $\bs{x}^\sigma \in \sqrt{M_P}$. Hence, $\bs{x}^\sigma \in \sqrt{I_\mc{T}+M_P}$, which proves (\ref{eq:to-prove}).
\end{proof}

The final ingredient for the boundary analogue of \Cref{cor:betke-mcmullen-spec} is the following correspondence between regular unimodular triangulations of $\partial P$ and squarefree initial ideals $\initial_{\prec_\omega}(I_{\partial P})$. It is our second main contribution in this section.

\begin{corollary}\label{cor:sturmfels-bd-unimod}
    Let $\Delta$ be a regular triangulation of $\partial P$ induced by $\omega_\partial$, and let $\omega$ be a $\omega_\partial$-compatible extension as in \Cref{th:sturmfels-bd-corr}. Then $\Delta$ is unimodular if and only if $\initial_{\prec_\omega}(I_{\partial P})$ is squarefree.
\end{corollary}

\begin{proof}
   $\Leftarrow$: Let $F$ be any facet of $P$. By \Cref{lem:restriction-J-facets},
   \[
   \initial_{\prec_\omega}(I_{\partial P})\cap R_F = \initial_{\prec_{\omega_F}}(I_F)
   \]
   and since $\initial_{\prec_\omega}(I_{\partial P})$ is squarefree, so is  $\initial_{\prec_{\omega_F}}(I_F)$. Now $\omega_F$ induces the regular triangulation $\Delta_{|F}$ of the lattice polytope $F$ and \Cref{cor:sturmfels-unimod} implies that $\Delta_{|F}$ is unimodular. Since this holds for every facet $F$ of $P$, every maximal simplex of $\Delta$ is unimodular. Thus $\Delta$ is unimodular.

   $\Rightarrow$: We show that every minimal monomial generator of $  \initial_{\prec_\omega}(I_{\partial P})$ is squarefree. Let $\bs{x}^{\bs{c}}=\prod_{\alpha\in \mathcal{A}}x_{\bs{\alpha}}^{c_{\bs{\alpha}}}$ be such a generator, and let $\sigma \coloneqq \{\bs{\alpha}:c_{\bs{\alpha}}\neq 0\} \subseteq \mathcal{A}$. We distinguish two cases:

   \noindent\textit{Case 1:} $\sigma$ is contained in a facet $F$ of $P$. Then $\bs{x}^{\bs{c}} \in \initial_{\prec_\omega}(I_{\partial P})\cap R_F$ and, by \Cref{lem:restriction-J-facets}, $\bs{x}^{\bs{c}}\in \initial_{\prec_{\omega_F}}(I_F)$. Since $\Delta_{|F}$ is unimodular, the classical \Cref{cor:sturmfels-unimod} implies that $\initial_{\prec_{\omega_F}}(I_F)$ is squarefree. Therefore $\bs{x}^{\bs{c}} $ is squarefree.

   \noindent\textit{Case 2:} $\sigma$ is not contained in any facet of $P$. Let $\bs{x}^\sigma = \prod_{\bs{\alpha}\in \sigma} x_{\bs{\alpha}}$. Since $\sigma$ is not contained in any facet, the argument from \Cref{lem:subdivision-facets} shows that 
   $$\bs{u}_\sigma = \sum_{\bs{\alpha} \in \sigma} (\bs{\alpha},1) \in S_P^\circ.$$
   Hence, $\bs{x}^\sigma \in M_P \subseteq \initial_{\prec_\omega}(I_{\partial P})$. Since $\bs{x}^\sigma$ divides $\bs{x}^{\bs{c}}$ by construction and since the latter is a minimal monomial generator of $\initial_{\prec_\omega}(I_{\partial P})$, we conclude $\bs{x}^{\bs{c}}= \bs{x}^\sigma$, which is squarefree.

   It follows that $\initial_{\prec_\omega}(I_{\partial P})$ is a squarefree monomial ideal.
\end{proof}

\noindent The proof of the special case of \Cref{th:stapledon} now follows like a breeze:

\begin{corollary}\label{cor:stapledon-spec}
    Let $P \subset \RR^d$ be a lattice $d$-polytope whose boundary complex $\partial P$ admits a regular unimodular triangulation $\Delta$. Then $h^\ast_{\partial P}(z) = h_\Delta(z)$.
\end{corollary}

\begin{proof}
    Let $\omega_\partial\in \RR^{|\partial P\cap \ZZ^d|}$ be a weight vector that induces the triangulation $\Delta$. 
    Since we proved that $I_{\partial P}$ is homogeneous (see \Cref{prop:general-prop} (iii)), we can apply \Cref{cor:hilb-series-hom-ideal} to obtain
    \begin{align}\label{eq:hilb-series-initial}
        \Hilb_{R_P/I_{\partial P}}(z) = \Hilb_{R_P/\initial_{\prec_\omega}(I_{\partial P})}(z),
    \end{align}
    where $\omega$ is a $\omega_\partial$-compatible extension. 
    \noindent In a final step, applying \Cref{th:sturmfels-bd-corr} and \Cref{cor:sturmfels-bd-unimod} leads to
    \[
    \Hilb_{R_P/\initial_{\prec_\omega}(I_{\partial P})}(z) = \Hilb_{\KK[\Delta]}(z) = \frac{h_\Delta(z)}{(1-z)^d}.
    \]
    Combining this with (\ref{eq:hilb-ehrhart-bdP}), (\ref{eq:hilb-series-hom-ideal-P}) and (\ref{eq:hilb-series-initial}) proves the statement.
\end{proof}

\begin{remark}
    Every  boundary-supported monomial in $M_P$ is divisible by a quadratic monomial $x_{\bs{\alpha}} x_{\bs{\beta}} \in M_P$, where $\bs{\alpha},\bs{\beta}$ do not lie in a common facet of $P$. Hence, the ideal $M_P$ is generated (up to saturation, radical and combinatorial control) by degree-$2$ generators of this form  as well as variables coming from interior lattice points.
\end{remark}

\begin{example}
    Consider $P = \conv\{(0,0),(2,0),(2,1),(1,2),(0,2)\}$. The contained lattice points are $a_1 = (0,0),\;a_2 = (1,0),\;a_3 = (2,0),\;a_4 = (2,1),\;a_5 = (1,2),\;a_6 = (0,2),\;a_7 = (0,1),\;a_8 = (1,1)$, with $a_8$ being the only interior lattice point. The toric component $I_P$ is generated by
    \begin{align*}
        I_P = \langle &x_1 x_3 - x_2^2, x_1 x_6 - x_7^2, \\ &x_1x_8-x_2x_7, x_2x_4-x_3x_8,x_5x_7-x_6x_8,x_2x_5-x_8^2,x_3x_6-x_8^2,x_4x_7-x_8^2, \\ &x_1x_4 -x_3x_7, x_2x_8-x_3x_7, x_3x_5-x_4x_8, x_4x_6-x_5x_8\rangle.
    \end{align*}
    Some of these relations are illustrated in \Cref{fig:Example-ToricBoundaryIdeal}. One can easily identify the quadratic generators of $M_P$ being
    \begin{align*}
        M_P = \langle x_8,x_1x_4, x_1x_5, x_2x_4,x_2x_5,x_2x_6,x_2x_7,x_3x_5,x_3x_6,x_3x_7,x_4x_6,x_4x_7,x_5x_7 \rangle.
    \end{align*}

    \begin{figure}[h]
        \centering
        
        % Row 1
        \begin{subfigure}{0.45\textwidth}
            \centering
            \begin{tikzpicture}[scale=1.8]
              % polygon fill
              \fill[gray!25] 
                (0,0) -- (2,0) -- (2,1) -- (1,2) -- (0,2) -- cycle;
              % polygon boundary
              \draw[thin,gray] 
                (0,0) -- (2,0) -- (2,1) -- (1,2) -- (0,2) -- cycle;
              % boundary lattice points
              \fill[black] (0,0) circle (1pt) node[left, font=\tiny] {$(0,0)$};
              \fill[black] (1,0) circle (1pt);
              \fill[black] (2,0) circle (1pt);
              \fill[gray] (2,1) circle (1pt);
              \fill[gray] (1,2) circle (1pt);
              \fill[gray] (0,2) circle (1pt);
              \fill[gray] (0,1) circle (1pt);
              % interior lattice point
              \fill[gray] (1,1) circle (0.9pt);

              % relations
              \draw[thin,brickred]
                (0,0) -- (1,0);
              \draw[thin,brickred]  
                (1,0) -- (2,0);
            \end{tikzpicture}
            \caption{$x_1 x_3 - x_2^2 \in I_P$.}
        \end{subfigure}
        \hfill
        \begin{subfigure}{0.45\textwidth}
            \centering
            \begin{tikzpicture}[scale=1.8]
              % polygon fill
              \fill[gray!25] 
                (0,0) -- (2,0) -- (2,1) -- (1,2) -- (0,2) -- cycle;
              % polygon boundary
              \draw[thin,gray] 
                (0,0) -- (2,0) -- (2,1) -- (1,2) -- (0,2) -- cycle;
              % boundary lattice points
              \fill[black] (0,0) circle (1pt) node[left, font=\tiny] {$(0,0)$};
              \fill[gray] (1,0) circle (1pt);
              \fill[gray] (2,0) circle (1pt);
              \fill[gray] (2,1) circle (1pt);
              \fill[gray] (1,2) circle (1pt);
              \fill[black] (0,2) circle (1pt);
              \fill[black] (0,1) circle (1pt);
              % interior lattice point
              \fill[gray] (1,1) circle (0.9pt);

              % relations
              \draw[thin,brickred]
                (0,0) -- (0,1);
              \draw[thin,brickred]  
                (0,1) -- (0,2);
            \end{tikzpicture}
            \caption{$x_1 x_6 - x_7^2 \in I_P$}
        \end{subfigure}
        
        \vspace{0.5cm}
        
        % Row 2
        \begin{subfigure}{0.45\textwidth}
            \centering
            \begin{tikzpicture}[scale=1.8]
              % polygon fill
              \fill[gray!25] 
                (0,0) -- (2,0) -- (2,1) -- (1,2) -- (0,2) -- cycle;
              % polygon boundary
              \draw[thin,gray] 
                (0,0) -- (2,0) -- (2,1) -- (1,2) -- (0,2) -- cycle;
              % boundary lattice points
              \fill[black] (0,0) circle (1pt) node[left, font=\tiny] {$(0,0)$};
              \fill[black] (1,0) circle (1pt);
              \fill[gray] (2,0) circle (1pt);
              \fill[gray] (2,1) circle (1pt);
              \fill[gray] (1,2) circle (1pt);
              \fill[gray] (0,2) circle (1pt);
              \fill[black] (0,1) circle (1pt);
              % interior lattice point
              \fill[black] (1,1) circle (0.9pt);

              % relations
              \draw[thin,brickred]
                (0,0) -- (1,1);
              \draw[thin,brickred]  
                (0,1) -- (1,0);
              \draw[thin,black]
                (0,0) -- (1,0) -- (1,1) -- (0,1) -- cycle;
            \end{tikzpicture}
            \caption{$x_1 x_8 - x_2 x_7 \in I_P$.}
        \end{subfigure}
        \hfill
        \begin{subfigure}{0.45\textwidth}
            \centering
            \begin{tikzpicture}[scale=1.8]
              % polygon fill
              \fill[gray!25] 
                (0,0) -- (2,0) -- (2,1) -- (1,2) -- (0,2) -- cycle;
              % polygon boundary
              \draw[thin,gray] 
                (0,0) -- (2,0) -- (2,1) -- (1,2) -- (0,2) -- cycle;
              % boundary lattice points
              \fill[gray] (0,0) circle (1pt) node[left, font=\tiny, black] {$(0,0)$};
              \fill[black] (1,0) circle (1pt);
              \fill[black] (2,0) circle (1pt);
              \fill[black] (2,1) circle (1pt);
              \fill[gray] (1,2) circle (1pt);
              \fill[gray] (0,2) circle (1pt);
              \fill[gray] (0,1) circle (1pt);
              % interior lattice point
              \fill[black] (1,1) circle (0.9pt);

              % relations
              \draw[thin,brickred]
                (1,0) -- (2,1);
              \draw[thin,brickred]  
                (1,1) -- (2,0);
              \draw[thin,black]
                (1,0) -- (2,0) -- (2,1) -- (1,1) -- cycle;
            \end{tikzpicture}
            \caption{$x_2 x_4 - x_3 x_8 \in I_P$}
        \end{subfigure}
        
        \vspace{0.5cm}
        
        % Row 3
        \begin{subfigure}{0.45\textwidth}
            \centering
            \begin{tikzpicture}[scale=1.8]
              % polygon fill
              \fill[gray!25] 
                (0,0) -- (2,0) -- (2,1) -- (1,2) -- (0,2) -- cycle;
              % polygon boundary
              \draw[thin,gray] 
                (0,0) -- (2,0) -- (2,1) -- (1,2) -- (0,2) -- cycle;
              % boundary lattice points
              \fill[black] (0,0) circle (1pt) node[left, font=\tiny] {$(0,0)$};
              \fill[gray] (1,0) circle (1pt);
              \fill[black] (2,0) circle (1pt);
              \fill[black] (2,1) circle (1pt);
              \fill[gray] (1,2) circle (1pt);
              \fill[gray] (0,2) circle (1pt);
              \fill[black] (0,1) circle (1pt);
              % interior lattice point
              \fill[gray] (1,1) circle (0.9pt);

              % relations
              \draw[thin,brickred]
                (0,0) -- (2,1);
              \draw[thin,brickred]  
                (0,1) -- (2,0);
              \draw[thin,black]
                (0,0) -- (2,0) -- (2,1) -- (0,1) -- cycle;
            \end{tikzpicture}
            \caption{$x_1 x_4 - x_3 x_7 \in I_P$}
        \end{subfigure}
        \hfill
        \begin{subfigure}{0.45\textwidth}
            \centering
            \begin{tikzpicture}[scale=1.8]
              % polygon fill
              \fill[gray!25] 
                (0,0) -- (2,0) -- (2,1) -- (1,2) -- (0,2) -- cycle;
              % polygon boundary
              \draw[thin,gray] 
                (0,0) -- (2,0) -- (2,1) -- (1,2) -- (0,2) -- cycle;
              % boundary lattice points
              \fill[gray] (0,0) circle (1pt) node[left, font=\tiny, black] {$(0,0)$};
              \fill[gray] (1,0) circle (1pt);
              \fill[gray] (2,0) circle (1pt);
              \fill[black] (2,1) circle (1pt);
              \fill[black] (1,2) circle (1pt);
              \fill[black] (0,2) circle (1pt);
              \fill[gray] (0,1) circle (1pt);
              % interior lattice point
              \fill[black] (1,1) circle (0.9pt);

              % relations
              \draw[thin,brickred]
                (0,2) -- (2,1);
              \draw[thin,brickred]  
                (1,1) -- (1,2);
              \draw[thin,black]
                (1,1) -- (2,1) -- (1,2) -- (0,2) -- cycle;
            \end{tikzpicture}
            \caption{$x_4 x_6 - x_5 x_8 \in I_P$}
        \end{subfigure}
        
        \caption{Illustration of binomial relations eliminated by component $I_P$ of toric boundary ideal $I_{\partial P}$ of $P = \conv\{(0,0),(2,0),(2,1),(1,2),(0,2)\}$.}
        \label{fig:Example-ToricBoundaryIdeal}
    \end{figure}
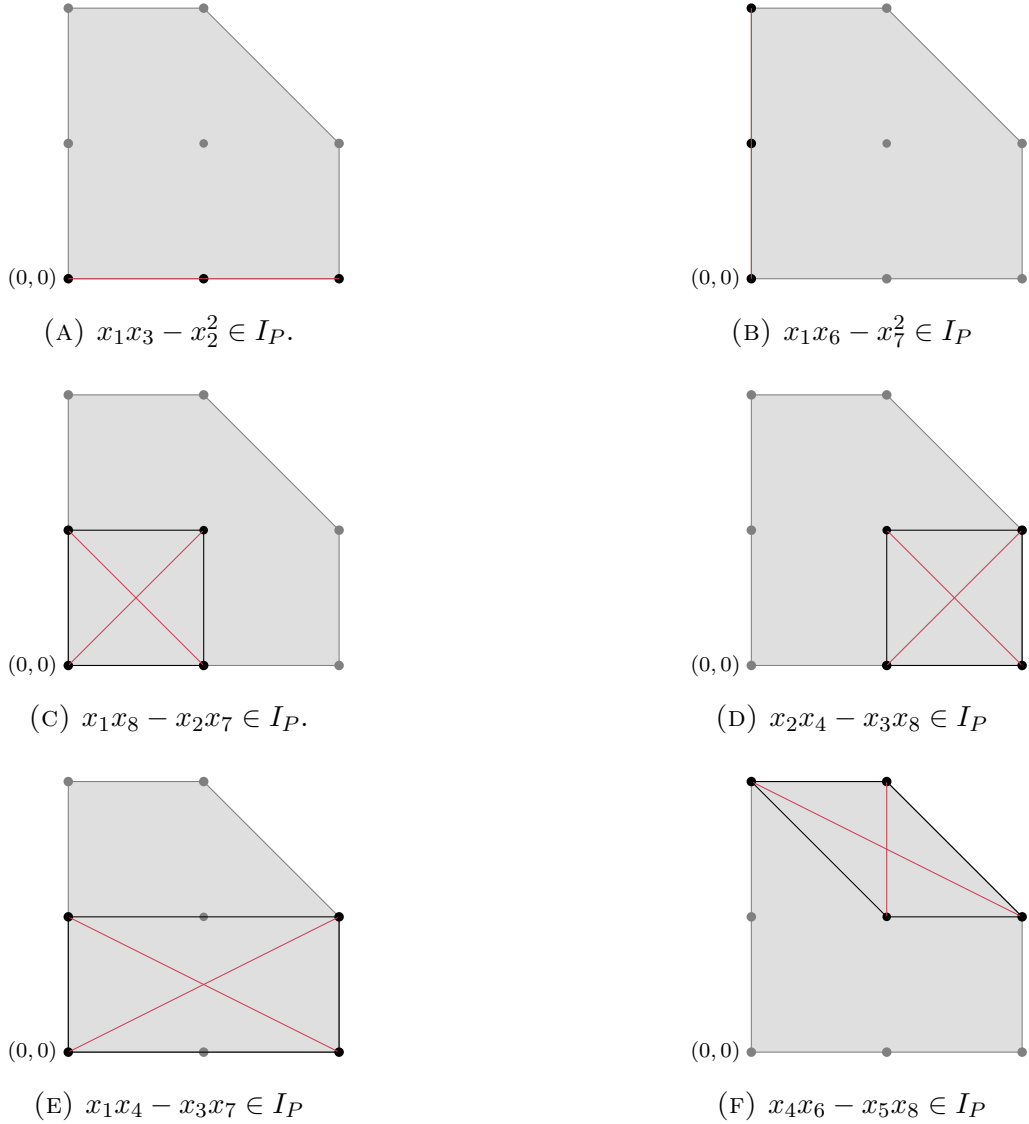
\end{example}

%%%--------------------------------------------------------------------------------%%%
%%%                           S E C T I O N     4                                  %%%
%%%--------------------------------------------------------------------------------%%%
\section{Results on $h^\ast(P)$} \label{sect:Applic} In this section, we contribute to the study of characterizations of the $h^\ast$-vector of lattice polytopes. First, we formulate Dehn-Sommerville relations between $h^\ast(P)$ and $h^\ast(\partial P)$ that hold without any additional assumptions for any lattice polytope $P$. Building on these relations, we establish a connection between face enumeration of simplicial manifolds and regularly unimodularly triangulated lattice polytopes. We first show an analogue of Klee's Dehn-Sommerville relations for simplicial balls. Using this and the results from the previous sections, we show that (under suitable assumptions) the vector $(h_0^\ast(P),h_1^\ast(P)-h_{d}^\ast(P),h_2^\ast(P)-h_{d-1}^\ast(P),\ldots,h_{\lfloor d/2\rfloor}^\ast(P)-h_{d+1-\lfloor d/2\rfloor}^\ast(P))$ is an $M$-sequence. We then characterize symmetry properties of $h^\ast(P)$ using the property that a polytopal realization of  a regular unimodular triangulation of $\partial P$ is stacked. Finally, we formulate upper and lower bounds for these differences under certain assumptions regarding the boundary.

%%%------------------------------ Subsection ----------------------------------------%
\subsection{Dehn-Sommerville relations for Ehrhart $h^\ast$-vectors} 
The goal of this section is to show a version of the classical Dehn-Sommerville relations for balls and spheres for $h^\ast$-vectors of lattice polytopes and their boundaries. 
The proof will go along similar lines as the proof of \cite[Theorem 3.1]{NovikSwartz09}, which might go back to Macdonald \cite{Macdonald1971}:

\begin{theorem}\label{th:ds-lattice-polytopes}
    Let $P \subset \RR^d$ be a lattice $d$-polytope. Then
    $$h_i^\ast(P) - h_{d-i+1}^\ast(P) = h_i^\ast(\partial P) - h_{i-1}^\ast(\partial P)$$
    for all $0\leq i \leq d$. 
\end{theorem}

\begin{proof}
    By Ehrhart-Macdonald reciprocity (see \cite[Theorem 4.4]{Beck2010})
    \begin{align}\label{eq:ehrhart-macdonald}
        \Ehr_P\left(\frac{1}{z}\right) = (-1)^{d+1} \Ehr_{P^\circ}(z).
    \end{align}
    The left-hand side of (\ref{eq:ehrhart-macdonald}) is equal to 
    \[
    \Ehr_P\left(\frac{1}{z}\right) = \frac{\sum_{i=0}^d h_i^\ast(P) \left(\frac{1}{z}\right)^i }{(1-z)^{d+1}} = \frac{\sum_{i=0}^d h_i^\ast(P) \left(z\right)^{d-i+1} }{(z-1)^{d+1}}.
    \]
    Using $\Ehr_{P^\circ}(z) = \frac{\sum_{i=0}^d h_i^\ast(\partial P) z^i}{(1-z)^{d+1}}$ and multiplying both sides of (\ref{eq:ehrhart-macdonald}) by $(-1)^{d+1}$, we obtain that
    \begin{align}\label{eq:identity-interior-P}
        \sum_{i=0}^d \left( h_{d-i+1}^\ast(P) - h_i^\ast(P^\circ) \right) z^i = 0 
    \end{align}
    implies that  $h_{d-i+1}^\ast(P) = h_i^\ast(P^\circ)$ for all $0 \leq i \leq d$. 

    Moreover, the Ehrhart series of $P$ can be written as
    \[
    \Ehr_P(z) = \Ehr_{P^\circ}(z) + \Ehr_{\partial P}(z).
    \]
    On the level of Ehrhart polynomials, this translates to
    \[
    h_P^\ast(z) = h_{P^\circ}^\ast(z) + (1-z)h_{\partial P}^\ast(z),
    \]
    leading to coefficient-wise relations
    \begin{align*}
        h_i^\ast(P) &= h_i^\ast(P^\circ) + h_i^\ast(\partial P) - h_{i-1}^\ast(\partial P) \\ &\overset{(\ref{eq:identity-interior-P})}{=} h_{d-i+1}^\ast(P) + h_i^\ast(\partial P) - h_{i-1}^\ast(\partial P),
    \end{align*}
    which shows the claim.
\end{proof}

An even shorter proof can be given in the special case when $P$ admits a regular unimodular triangulation.

\begin{corollary}\label{cor:ds-lattice-polytopes-reguni}
    Let $P \subset \RR^d$ be a lattice $d$-polytope admitting a regular unimodular triangulation. Then
    \[
    h_i^\ast(P) - h_{d-i+1}^\ast(P) = h_i^\ast(\partial P) - h_{i-1}^\ast(\partial P)
    \]
    for all $0\leq i \leq d$. 
\end{corollary}

\begin{proof}
    Let $\mc{T}$ be a regular unimodular triangulation of $P$ and let $\Delta = \partial \mc{T}$ denote the induced triangulation of $\partial P$ obtained by restricting $\mc{T}$ to the boundary complex of $P$. Then 
    \begin{align*}
        h_i^\ast(P) - h_{d-i+1}^\ast(P) &= h_i(\mc{T})-h_{d-i+1}(\mc{T}) \\ &= h_i(\Delta) - h_{i-1}(\Delta) \\ &= h_i^\ast(\partial P) - h_{i-1}^\ast(\partial P)
    \end{align*}
    where the first and third equality are simply \Cref{cor:betke-mcmullen-spec} and \Cref{cor:stapledon-spec}, respectively. The second identity follows from Theorem 3.1 in \cite{NovikSwartz09} by observing that $\mc{T}$ is a $\KK$-homology $d$-ball.
\end{proof}

%%%------------------------------ Subsection ----------------------------------------%
\subsection{Properties of coefficient-wise differences in $h^\ast(P)$ and $h^\ast(\partial P)$} \label{subsect:differences} 
For the sake of readability we introduce a shorthand notation:
\begin{align*}
    g^\ast_i(\partial P) \coloneqq h^\ast_i(\partial P) - h^\ast_{i-1}(\partial P)  \text{ for }0\leq i\leq \lfloor\frac{d}{2}\rfloor,
\end{align*}
where $P\subset \RR^d$ is a lattice $d$-polytope. It follows from \Cref{th:ds-lattice-polytopes} that, alternatively, we have
$g^\ast_i(\partial P)=h_i^\ast(P)-h_{d-i+1}^\ast(P)$. 

%A lattice $d$-polytope $P \subset \RR^d$ is a compact $d$-dimensional manifold with boundary, homeomorphic to $\BB^d$. The boundary $\partial P$ is homeomorphic to $\sphere^{d-1}$ and a $(d-1)$-dimensional manifold without boundary. A triangulation $\mc{T}$ of $P$ is a simplicial $d$-ball, while $\partial \mc{T} = \Delta$ is a simplicial $(d-1)$-sphere. For more (co)homology background on these definitions we refer to \cite{KleeNovik16}.

The main bridge we use was kindly communicated to us by Francisco Santos and is the subject of the following lemma:

\begin{lemma}\label{claim:santos}
    Let $P$ be a lattice $d$-polytope, whose boundary $\partial P$ admits a regular triangulation $\Delta$. Then $\Delta$ is polytopal, i.e., $\Delta$ is combinatorially equivalent to the boundary complex of a simplicial polytope $Q$. 
\end{lemma}

\begin{proof}
    Consider the following auxiliary triangulation $\mc{T}^\prime$ of $P$: add an extra point $\bs{p}$ in the interior of $P$, and define $\mc{T}^\prime$ to be the cone of $\partial \mc{T}$ with apex at $\bs{p}$. $\mc{T}^\prime$ is indeed a regular triangulation of $P$ since one can alternatively describe it as follows: first, take the cone over the boundary complex $\partial P$ with apex $\bs{p}$. This is a pulling refinement of the trivial subdivision of $P$ (see \cite[Section 4.3.4]{DeLoeraRambauSantos2010}). Then, refine this subdivision using the same lifting weights that give $\mc{T}$ as a regular triangulation. That this recovers $\mc{T}^\prime$ and that it produces a regular subdivision is Lemma 2.3.26 in \cite{DeLoeraRambauSantos2010}. 

    Now denote by $\mathcal{A}$ the point configuration underlying  $\mc{T}^\prime$, i.e., $\mathcal{A}$ is the vertex set of $\mc{T}$ together with the additional point $\bs{p}$. Then one notes that $\partial \mc{T}$ equals the link of $\bs{p}$ in $\mc{T}^\prime$, which is  a regular subdivision of $\mathcal{A}\setminus\{ \bs{p}\}$ by \cite[Lemma 4.2.20]{DeLoeraRambauSantos2010}, and that regular triangulations of totally cyclic configurations such as $\mathcal{A}\setminus \{\bs{p}\}$ are polytopal complexes (see \cite[Lemma 9.5.1]{DeLoeraRambauSantos2010}). This proves the claim.
\end{proof}

We are now able to translate some prime results from face enumeration for simplicial manifolds, as the $g$-theorem (see \cite{BilleraLee1981,Stanley1980}) and the Generalized Lower Bound Theorem for simplicial polytopes (see \cite{McMullenWalkup1971, MuraiNevo2013}) to the setting of $h^\ast$-vectors of polytopes and their boundaries. We start by introducing the concept of an $M$-sequence of integers. Given positive integers $\nu$ and $i$, there exists a unique decomposition (also known as \emph{$i$-binomial expansion}) of the form
\begin{align*}
    \nu = \binom{a_i}{i} + \binom{a_{i-1}}{i-1} + \cdots + \binom{a_j}{j}, \hspace{1.5 em} \text{where } a_i > a_{i-1} > \cdots > a_j \geq j > 0.
\end{align*}
Define
\begin{align*}
    \nu^{\langle i\rangle} \coloneqq \binom{a_i + 1}{i+1} + \binom{a_{i-1}+1}{i} + \cdots + \binom{a_j +1}{j+1}.
\end{align*}
A sequence of integers $(1,\nu_1,\nu_2,\dots)$ is an \emph{$M$-sequence} if $0 \leq \nu_{i+1} \leq \nu_i^{\langle i\rangle}$ for all $i$. Equivalently, an $M$-sequence is the Hilbert function as a standard graded Artinian $\KK$-algebra (see e.g., \cite{StanleyGreenBook}).

 %When referring to the $h$-vector of a homology manifold or a simplicial polytope $\mc{K}$ in the following, we always mean the $h$-vector of its boundary complex $\partial \mc{K}$.
\begin{theorem}\label{th:g-theorem-transfer}
    Let $P \subset \RR^d$ be a lattice $d$-polytope, whose boundary complex $\partial P$ admits a regular unimodular triangulation. Then
    \begin{enumerate}
        \item[(i)] $h_j^\ast(\partial P) = h_{d-j}^\ast(\partial P)$ for all $j$,
        \item[(ii)] $1 = h_0^\ast(\partial P) \leq h_1^\ast(\partial P) \leq \cdots \leq h_{\lfloor d/2 \rfloor}^\ast(\partial P)$, and
        \item [(iii)] the sequence $g^\ast(\partial P) \coloneqq (g_0^\ast(\partial P),g_1^\ast(\partial P),\dots,g_{\lfloor d/2\rfloor}^\ast(\partial P))$ is an $M$-sequence.
    \end{enumerate}
\end{theorem}

\begin{proof}
    Let $\Delta$ be a regular unimodular triangulation of $\partial P$. Due to \Cref{claim:santos}, $\Delta$ is combinatorially equivalent to $\partial \mc{K}$ for a simplicial $d$-polytope $\mc{K}$, which implies $h(\Delta) = h(\partial \mc{K})$. By \Cref{th:g-theorem} and \Cref{cor:stapledon-spec}, all the properties (i)-(iii) are translated to $h^\ast(\partial P)$ via 
    \[
    h(\partial \mc{K}) = h(\Delta) = h^\ast(\partial P).
    \]
\end{proof}

\begin{remark}
    Symmetry of $h(\partial \mc{K})$ holds more generally when $\partial \mc{K}$ is a $(d-1)$-dimensional homology sphere. This is an immediate consequence of Klee's Dehn-Sommerville equations \cite{Klee1964}. Both (i) and (ii) in \Cref{th:g-theorem-transfer} were already (implicitly) stated by Stapledon \cite{Stapledon2009}.
\end{remark}

\Cref{th:g-theorem-transfer} (iii) means that the growth rate between consecutive entries of $h^\ast(\partial P)$ is bounded. Recall that by \Cref{th:g-theorem-transfer}, unimodality of $h^\ast(\partial P)$ implies 
\[
g_j^\ast(\partial P) = h_j^\ast(\partial P) - h_{j-1}(\partial P) \geq 0.
\]
The Dehn-Sommerville equations stated in \Cref{cor:ds-lattice-polytopes-reguni} then directly imply that 
\begin{equation}\label{eq:topheavy}
h_i^\ast(P) \geq h_{d-i+1}^\ast(P), 
\end{equation}
a property also referred to as \emph{top-heaviness}. This has previously been proven  by Athanasiadis \cite{Athanasiadis}. 
The goal of this section is to investigate how the cases of equality in \eqref{eq:topheavy} can be characterized. A treatment of these cases is given by the Generalized Lower Bound Theorem (see \Cref{th:glbt}). We add the following definition:

A simplicial $d$-polytope $\mc{K}$ is called \emph{$(r-1)$-stacked} if there exists a triangulation $\mc{T}$ of $P$ such that $\partial \mc{T} = \partial P$ and $\Skel_{d-r}(\mc{T}) = \Skel_{d-r}(P)$, where $\Skel_{d-r}(P)$ (resp. $\Skel_{d-r}(\mc{T})$) denotes the $(d-r)$-skeleton of $P$ (resp. $\mc{T})$, i.e., the set of faces of $P$ (resp. $\mc{T}$) of dimension at most $d-r$. In other words, $\mc{T}$ extends the triangulation of $\partial P$ to the interior of $P$ without introducing any interior faces of dimension $\leq d-r$. Moreover, by definition, $(r-1)$-stackedness of $\partial P$ implies that $\partial P$ is $\ell$-stacked for all $r \leq \ell \leq d$.

The key ingredient to characterize equality in \eqref{eq:topheavy} is the observation that due to the Dehn-Sommerville relations, we have $h^\ast_r(P) = h^\ast_{d-r+1}(P)$ if and only if $h^\ast_{r-1}(\partial P)=h_r^\ast(\partial P)$, which can be characterized via the generalized lower bound theorem. More precisely, we derive the following characterization:

\begin{theorem}\label{th:glbt-transfer}
    Let $P \subset \RR^d$ be a lattice $d$-polytope, whose boundary complex $\partial P$ admits a regular unimodular triangulation $\Delta$ and let $\mc{K}$ be a simplicial polytope for which $\Delta \cong \partial\mc{K}$. Then, $\mc{K}$ is $(r-1)$-stacked for some $r \leq \lfloor d/2\rfloor$ if and only if
    \[
    h_j^\ast(P) = h_{d-j+1}^\ast(P) \hspace{1.5 em} \text{for all } r \leq j \leq \lfloor d/2\rfloor.
    \]
    In particular, $h_j^\ast(\partial P) = h_{j-1}^\ast(\partial P)$ for $r \leq i \leq d-r+1$ if and only if $\mc{K}$ is $(r-1)$-stacked for some $r \leq \lfloor d/2 \rfloor$. 
\end{theorem}

\begin{proof}
    We start by noting that, by definition, $(r-1)$-stackedness of $\mc{K}$ implies that $\mc{K}$ is $\ell$-stacked for $r \leq \ell \leq \lfloor d/2 \rfloor$.

    $\Rightarrow$: If $\mc{K}$ is $(r-1)$-stacked, then we get the following chain of identities:
    \begin{align*}
        h_r^\ast(\partial P) &\overset{(1)}{=} h_r(\Delta) \overset{(2)}{=} h_r(\partial \mc{K}) \overset{(3)}{=} h_{r-1}(\partial \mc{K}) \\ &\overset{(4)}{=} h_{r-1}(\Delta) \overset{(5)}{=} h_{r-1}^\ast(\partial P).
    \end{align*}
    The equalities (1) and (5) follow by \Cref{cor:stapledon-spec}, (2) and (4) are due to $\Delta \cong \partial \mc{K}$ and (3) is precisely the GLBT (\Cref{th:glbt}). Applying the Dehn-Sommerville relations in \Cref{th:ds-lattice-polytopes}) and the fact from the beginning of this proof, we obtain 
    \[
    h_j^\ast(P) = h_{d-j+1}^\ast(P) \hspace{1em} \text{ for all } r \leq j \leq \lfloor d/2 \rfloor.
    \]

    $\Leftarrow$: Reversing the argumentation, we arrive at $h_r(\partial \mc{K}) = h_{r-1}(\partial \mc{K})$, which implies that $\mc{K}$ is $(r-1)$-stacked by \Cref{th:glbt}. 
\end{proof}

We note that \Cref{th:glbt-transfer} also shows that the symmetry $h_j^\ast(P) = h_{d-j+1}^\ast(P)$ propagates, i.e., if $h_j^\ast(P) = h_{d-j+1}^\ast(P)$ for some $j$, then $h_k^\ast(P) = h_{d-k+1}^\ast(P)$ for all $j\leq k\leq \lfloor \frac{d}{2}\rfloor$.

We illustrate this behavior with an example.
\begin{example}
    Let $\bs{a} = (N-1,\dots,N-1,N) \in \NN^d$. The 1-row Hermite normal form simplex (HNFS, for short) associated to $\bs{a}$ is defined as
    \[
    S_{\bs{a}} \coloneqq \text{conv}\{0,\bs{e}_1,\dots,\bs{e}_{d-1},\bs{a}\}.
    \]
    Figure \ref{fig:Examples-HNFS} shows realizations of $S_{\bs{a}}$ for low-dimensional choices of $\bs{a}$. A result from \cite{BruckampCaicedoJuhnke} shows that $S_{\bs{a}}$ admits a regular unimodular triangulation $\mc{T}$ if and only if $N \in \{k d, k d + 1\}$ and it is shown that
    \[
    h^\ast_{S_{\bs{a}}}(z)=\begin{cases}
        1+\sum_{j=1}^{d-1}kz^j+(k-1) z^d, \quad&\text{ if }N=kd\\
        1+\sum_{j=1}^{d}kz^j, \quad&\text{ if } N=kd+1.
    \end{cases}
    \]
    E.g., for $d=6,k=2$ and $N=k d + 1 = 13$, we obtain
    \[
    h^\ast_{S_{\bs{a}}}(z) = 1 + 2 \sum_{i=1}^d z^i,
    \]
    so $h^\ast(S_{\bs{a}}) = (1,2,2,2,2,2,2)$, satisfying $h^\ast_i(S_{\bs{a}}) = h^\ast_{6-i+1}(P)$ for $1 \leq i \leq 3$. 
    It is easy to see that for $N\in \{kd,kd+1\}$, the only lattice points on the boundary of the simplex $S_{\bs{a}}$ are its vertices and one additional point in the first case. Hence, the polytopal realizations of the boundary are given by two simplices glued together in the first case (which is $1$-stacked) and a simplex in the second case (which can be interpreted as $0$-stacked). From this we can deduce the above formula for the $h^\ast$-polynomial via \Cref{th:glbt-transfer}, where in the first case, for the last coefficient we additionally use that the normalized volume of $S_{\bs{a}}$ equals $N$.
    \begin{figure}[h]
        \begin{minipage}{0.45\textwidth}
            \centering
            \includegraphics[width=0.8\textwidth]{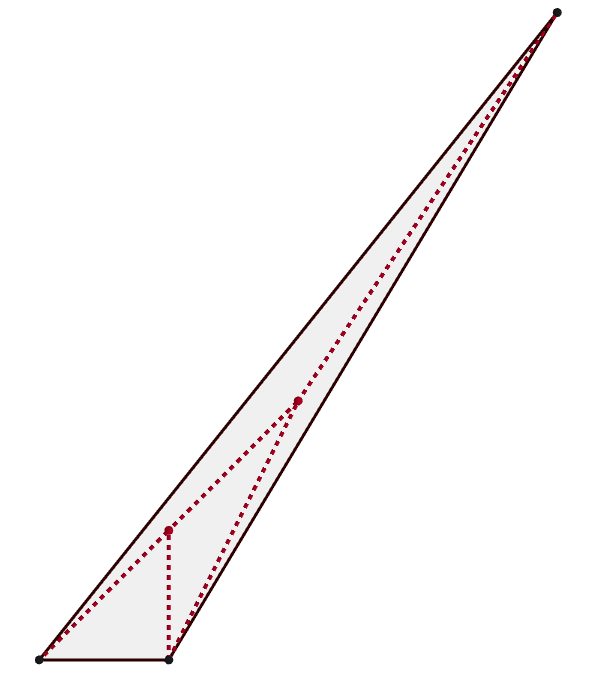}
        \end{minipage}
        \hspace{1em}
        \begin{minipage}{0.45\textwidth}
            \centering
            \includegraphics[width=0.65\textwidth]{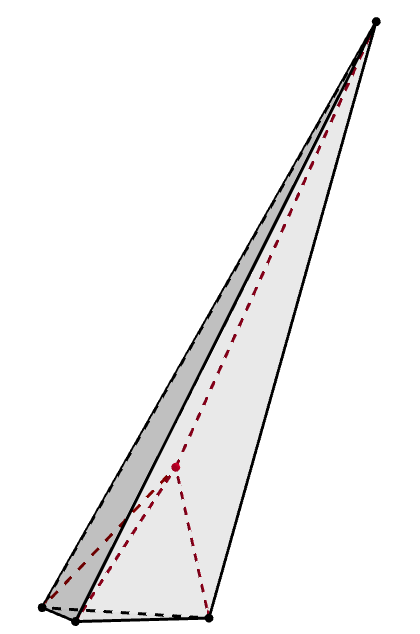}
        \end{minipage}
        \caption{1-row Hermite normal form simplices with $(d,k,N) = (2,2,5)$ and $(d,k,N) = (3,1,4)$.}
        \label{fig:Examples-HNFS}
    \end{figure}
\end{example}

\begin{remark}
    Recently, Bajo and Beck showed that $h^\ast(P) = h^\ast(\partial P)$ if $P$ is reflexive (see Corollary 5.4 in \cite{BajoBeck2023}).  If $\partial P$ additionally admits a regular unimodular triangulation, \Cref{th:glbt-transfer} imposes very strong conditions on the form of $h^\ast(P)$. More precisely, we have the following corollary.
\end{remark}

\begin{corollary}\label{cor:refl}
 Let $P \subset \RR^d$ be a reflexive lattice $d$-polytope, whose boundary complex $\partial P$ admits a regular unimodular triangulation. If  $h_j^\ast(P)=h_{d-j+1}^\ast(P)$ for some $1\leq j\leq$, then 
 \[
 h^\ast_j(P)=h_{j+1}^\ast(P)=\cdots = h_{d-j}^\ast(P)=h_{d-j+1}^\ast(P).
 \]
\end{corollary}
Intuitively, this corollary states that if a reflexive lattice polytope such that its boundary admits a regular unimodular triangulation attains equality in \eqref{eq:topheavy} for some $j$, then its $h^\ast$-vector stays constant between $h_j^\ast(P)$ and $h_{d-j+1}^\ast(P)$.

\begin{proof}
The statement directly follows from the characterization of reflexivity, due to Hibi \cite{Hibi}. Namely, $P$ is reflexive, if and only if $h_j^\ast(P)=h_{d-j}^\ast (P)$ for $0\leq j\leq d$. This implies
\[
 h^\ast_j(P)= h_{d-j}^\ast(P)= h_{j+1}^\ast(P)=h_{d-j-1}^\ast(P)= h_{j+1}^\ast(P)=\cdots,
\]
where the equalities follow from the alternating use of reflexivity and the assumption combined with \Cref{th:glbt-transfer}. 
\end{proof}

%%%------------------------------ Subsection ----------------------------------------%
\subsection{Upper and lower bounds on $g_i^\ast(P)$} In the final section, we derive an upper and lower bound for the differences $g_i^\ast(P) = h_i^\ast(P) - h_{d-i+1}^\ast(P)$, which hold under the assumption that $\partial P$ has a regular unimodular triangulation $\Delta$. One the one hand, the upper bound is established by a transfer of the $g$-theorem/Macaulay growth from boundary triangulations to asymmetry of the $h^\ast$-vector of $P$. On the other hand, the lower bound holds if $\Delta$ is assumed to be balanced, which strengthens the GLBT to some normalized inequalities.

The upper bound is given in terms of the boundary $h^\ast$-vector, more precisely by $h_1^\ast(\partial P)$:
\begin{proposition}\label{prop:upper-bound}
     Let $P \subset \RR^d$ be a lattice $d$-polytope, whose boundary complex $\partial P$ admits a regular unimodular triangulation $\Delta$. Then for every $1 \leq i \leq \lfloor d/2 \rfloor$,
    \[
    0 \leq g_i^\ast(P)=h_i^\ast(P) - h_{d-i+1}^\ast(P) \leq \binom{h_1^\ast(\partial P) + i-2}{i}.
    \]
    More generally, for $1 \leq k \leq \ell \leq \lfloor d/2 \rfloor$,
    \[
    0 \leq \sum_{i=\ell}^k (h_i^\ast(P) - h_{d-i+1}^\ast(P)) \leq \binom{h_1^\ast(\partial P) + k-1}{k} - \binom{h_1^\ast(\partial P) + \ell -2}{\ell-1}.
    \]
\end{proposition}

\begin{proof}
    As $g_i^\ast(\partial P)$ is an M-sequence according to \Cref{th:g-theorem-transfer}, Macaulay's theorem \cite{Macaulay1927} gives 
    \begin{align}\label{eq:Upper-Bound-1}
        h_i^\ast(P)-h_{d-i+1}^\ast(P) = h_i^\ast(\partial P) - h_{i-1}^\ast(\partial P) \leq \binom{g_1^\ast(\partial P) +i-1}{i}.
    \end{align}
    Inserting $g_1^\ast(\partial P) = h_1^\ast(\partial P) -1$ shows the upper bound.

    Summing over the left-hand side of (\ref{eq:Upper-Bound-1}) leads to
    \begin{align*}
        \sum_{i=\ell}^k (h_i^\ast(P)-h_{d-i+1}^\ast(P)) &\leq \sum_{i=\ell}^k \binom{h_1^\ast(\partial P) +i-2}{i} \\ &\overset{(*)}{\leq} \binom{h_1^\ast(\partial P)+k-1}{k} - \binom{h_1^\ast(\partial P) + \ell -2}{\ell-1}
    \end{align*}
    for $1 \leq \ell \leq k \leq \lfloor d/2 \rfloor$. Note that in $(*)$, we used the Hockey-Stick identity for binomial coefficients \cite{Jones1996}. 
\end{proof}

Before formulating the lower bound, we need the following definition. A $d$-dimensional simplicial complex $\mc{K}$ is called \emph{balanced} if its 1-skeleton (viewed as a graph) admits a proper $(d+1)$-coloring. A standard example is the barycentric subdivision of a simplicial complex. For further reading we refer to \cite[Section 5.1]{KleeNovik16}. 

The lower bound is expressed in terms of the coefficients of $h^\ast(\partial P)$ as well:
\begin{proposition}\label{prop:lower-bound}
    Let $P \subset \RR^d$ be a lattice $d$-polytope, whose boundary complex $\partial P$ admits a regular unimodular triangulation $\Delta$. 
    If, in addition, $\Delta$ is balanced, then
    \[
    h_i^\ast(P) - h_{d-i+1}^\ast(P) \geq \frac{d-2i+1}{i} h_{i-1}^\ast(\partial P).
    \]
\end{proposition}

\begin{proof}
    Under the additional assumption of $\Delta$ being balanced, the Generalized Lower Bound Theorem can be  strengthened to the normalized inequalities
    \[
    \frac{h_0(\Delta)}{\binom{d}{0}} \leq \frac{h_1(\Delta)}{\binom{d}{1}} \leq \cdots \leq \frac{h_i(\Delta)}{\binom{d}{i}} \leq \cdots \leq \frac{h_{\lfloor d/2 \rfloor}(\Delta)}{\binom{d}{\lfloor d/2 \rfloor}},
    \]
    (see \cite{MuraiJuhnke,KleeNovik})
    or equivalently,
    \[
    \Tilde{g}_i(\Delta) \coloneqq i h_i(\Delta) - (d-i+1)h_{i-1}(\Delta) \geq 0
    \]
    for $1\leq i\leq \lfloor \frac{d}{2}\rfloor$. 
    Comparing with $g_i^\ast(P) = h_i^\ast(P) - h_{d-i+1}^\ast(P) = h_i(\Delta) - h_{i-1}(\Delta)$, we obtain
    \[
    \Tilde{g}_i(\Delta) = ig_i^\ast(P) - (d-2i+1)h_{i-1}(\Delta),
    \]
    which leads to
    \begin{align*}
        h_i^\ast(P) - h_{d-i+1}^\ast(P) \geq \frac{d-2i+1}{i}h_{i-1}(\Delta) = \frac{d-2i+1}{i}h^\ast_{i-1}(\partial P). 
    \end{align*}
\end{proof}

%----------------------- Acknowledgments -----------------------------------------%
\noindent \textbf{Acknowledgments.} We would like to thank Francisco Santos for clearing up some confusion which occurred to us when connecting regular triangulations to polytopality. We are also grateful to our dear colleague Jhon Caicedo, who provided us with his code for generating examples of 1-row Hermite normal form simplices. We would also like to thank Matthias Beck who suggested to investigate the Ehrhart theory of polytope boundaries in the first place. This work was supported by the German Research Foundation (DFG) via SPP 2458 \emph{Combinatorial Synergies}, project number 539849618.

\bibliographystyle{plain}
\bibliography{references.bib}

\end{document}